\RequirePackage{fix-cm}
\documentclass[smallcondensed]{article}   % onecolumn (ditto)

\usepackage{graphicx}

\usepackage[pdfencoding=auto]{hyperref}
 \usepackage{amsfonts,amsmath,amssymb,amsthm}
\usepackage{amsbsy}
\usepackage{mathrsfs,mathtools,stmaryrd,wasysym}
\usepackage{enumerate}
\usepackage{xspace}
\usepackage{xcolor}
\usepackage{accents}
\usepackage{tcolorbox}
\usepackage[numbers]{natbib}
\usepackage[cal=cm,calscaled=.96,bb=ams,bbscaled=.96]{mathalpha}
\usepackage{caption}
\usepackage{subcaption}
\usepackage{./jscsty}

\usepackage{tikz}
\usetikzlibrary[calc]
\usepackage{tkz-euclide}
\usepackage{pdfsync}
\usepackage{array}
\usepackage{multicol}
\usepackage{ifthen}

\usepackage{pgfplots}
\pgfplotsset{compat=1.13}
\usetikzlibrary{spy}
\usepgfplotslibrary{units}
\usepackage{pgfplotstable}
\usepgfplotslibrary{groupplots}

\usetikzlibrary{external}
\tikzset{external/only named=true}

\usepackage{siunitx}

\ifdefined\unit\else
  \ifdefined\NewCommandCopy
    \NewCommandCopy\unit\si
  \else
    \NewDocumentCommand\unit{O{}m}{\si[#1]{#2}}
  \fi
\fi

\definecolor{darkblue2}{rgb}{0,0,.65}

\title{A weighted hybridizable discontinuous Galerkin method for drift-diffusion
problems \thanks{This work was partially supported by funds from GNCS-INdAM
``Professori Visitatori Bando 2020'', project ``A comparison between finite
volume and hybridizable discontinuous Galerkin methods for the simulation of
micro- and nano-electronic devices'' as well as the Leibniz competition 2020.} }
\author{Wenyu Lei$^1$        \and
        Stefano Piani$^2$    \and
        Patricio Farrell$^3$ \and
        Nella Rotundo$^4$    \and
        Luca Heltai$^2$      
}

\date{
$^1$  University of Electronic Science and Technology of China\\
	No.2006, Xiyuan Ave, West Hi-Tech Zone\\
	611731 Chengdu, China \\ \texttt{wenyu.lei@uestc.edu.cn}\\
$^2$  SISSA -- International School for Advance Studies\\
via Bonomea 265, 34136 Trieste, Italy\\
              \texttt{\{stefano.piani, luca.heltai\}@sissa.it} \\
$^3$  Weierstrass Institute Berlin\\Mohrenstr. 39, 10117 Berlin, Germany\\
              \texttt{patricio.farrell@wias-berlin.de} \\
$^4$   DIMAI -- University of Florence\\Viale Morgagni 67/A, 50134 Florence\\
              \texttt{nella.rotundo@unifi.it}\\[2ex]
              \today           
}

\begin{document}

\maketitle

\begin{abstract}
In this work we propose a weighted hybridizable discontinuous Galerkin method
(W-HDG) for drift-diffusion problems. By using specific exponential weights when
computing the $L^2$ product in each cell of the discretization, we are able to
mimic the behavior of the Slotboom  variables, and eliminate the
drift term from the local matrix contributions, while still solving the problem
for the primal variables. We show that the proposed
numerical scheme is well-posed, and validate numerically that it has the same
properties as classical HDG methods, including optimal convergence, and
superconvergence of postprocessed solutions. For polynomial degree zero,
dimension one, and vanishing HDG stabilization parameter, W-HDG coincides with
the Scharfetter-Gummel finite volume scheme (i.e., it produces the
same system matrix). 
The use of local exponential weights generalizes the
Scharfetter-Gummel scheme (the state-of-the-art for finite volume
discretization of transport dominated problems) to arbitrary high order
approximations.

\paragraph*{Keywords:~} Finite element methods; discontinuous Galerkin methods;
hybrid methods \and weighted norms; exponential fitting methods;
convection-diffusion equations; drift-diffusion problems
% % \PACS{PACS code1 \and PACS code2 \and more}
% \subclass{MSC 65N30 \and 65N12}
\end{abstract}

\section{Introduction}
Drift-diffusion problems are ubiquitous in nature, and describe processes in
which transport and diffusion play an equally important role. A notable example
is the van Roosbroeck system for the simulation of the electric potential and
charge carrier densities in semiconductor devices \cite{Mock1972}, where these
systems are used to model the charge transport of electron or hole densities in
a self-consistent electrical field.

Such systems are especially difficult to solve when advection dominates
diffusion \cite{Morton1996}, leading to stiff boundary layers that need to be
resolved appropriately by numerical discretization schemes. Moreover,
positivity preserving schemes are important when some of the unknown variables
represent a density, as it is the case for the van Roosbroeck system.

A standard technique for the simulation of drift-diffusion problems is the two-point
finite volume (TPFV) method in combination with a Scharfetter--Gummel flux discretization
\cite{Scharfetter1969} --- a finite volume variant of the
Il'in--Allen--Southwell scheme \cite{roos:stynes:tobiska}.
It has become a key tool
for the simulation of semiconductor devices \cite{Farrell2017, Scharfetter1969, Bank1987, Gaertner2015,Chainais-Hillairet2015, Ten-Thije-Boonkkamp1993, Abdel2021}, thanks to its positivity preserving
property (cf. \cite{bessemoulin2012finite}), and to the availability of excellent mesh generators for complicated
2D \cite{triangle} and 3D \cite{si:2015} domains, which allow to construct
unstructured Delaunay--Voronoi meshes (though anisotropic meshes in 3D remain a
challenge), capable of capturing boundary layers and other important features of
the simulation.

% To handle boundary layers, 
% must yield 

% It is worth pointing out that there
% are also finite element based simulation tools which take into account the
% multiscale nature of optoelectronic devices \cite{Maur2008TiberCAD}.

One of the biggest drawbacks of this two-point finite volume method is its inherent low
order nature. An interesting attempt to construct uniformly convergent second
order TPFV schemes for applications to semiconductor devices and plasma
physics was undertaken by the group of ten Thije Boonkkamp. In a series of
papers \cite{hbm:98,boonkkamp:05,Thije-Boonkkamp2011,Liu2013}, he and his
coworkers have considerably extended a uniformly convergent second order finite
difference scheme originally used by G.~D.~Thiart \cite{thiart:90}. In
\cite{boonkkamp:05}, this approach was called {\em complete flux scheme} (CFS),
owing to the fact that CFS has an additional source term contribution in the
Scharfetter--Gummel flux of the differential operator.

This scheme was later extended to nonuniform meshes in \cite{Farrell2017b},
where also uniform second-order convergence was proven. However, it is not clear
how to extend these ideas to produce finite volume schemes with convergence
rates higher than two. Finite volumes, even when using Scharfetter--Gummel
stabilization or the complete flux scheme, remain inherently low order methods.

Discontinuous Galerkin (DG) methods \cite{ArnoldBrezziCockburnMarini-2001-a} are
often thought to be a (possibly higher order) alternative to finite volume
approximations, but are rarely adopted outside of the academic community. The
combination of positivity preserving and the Scharfetter--Gummel stabilization
makes the finite volume approximation very robust, and the \emph{go-to} solution
for many drift-diffusion approximations.
DG methods have also been criticized due to their relatively high number of globally coupled degrees
of freedom. The hybridizable discontinuous Galerkin (HDG) alleviates this problem by 
discretizing the exact solution written in terms of many local problems as well as one single 
global problem. 

Within the finite volume community, hybrid high-order (HHO) 
methods have been studied 
for general advection-diffusion equations \cite{Di-Pietro2015}. It turns out that HHO can be interpreted as a super-convergent HDG method. Moreover, the lowest order HHO methods are 
known as hybrid finite volume (HFV) schemes \cite{Eymard2009, DRONIOU2010}, which 
aim at overcoming the limitations of classical two-point flux schemes. In 
\cite{Chainais-Hillairet2022}, three different types of HFV were studied, one of which 
is nonlinear in nature and is shown to preserve positivity.

% \textcolor{blue}{PF: We should may be make the connection from DG to HDG. And then point out what I sent some time ago via email:
% %
% \url{https://epubs.siam.org/doi/epdf/10.1137/140993971}
% %
% Hybrid high-order (HHO) methods seem to be related to HDG methods. The lowest order of the HHO methods is the hybrid finite volume (HFV) method which were studied in 
% %
% \url{https://link.springer.com/article/10.1007/s00211-022-01289-w}
% %
% We should point out the connections between HDG and HFV/HHO and cite at least both papers above. We could do this once we receive the revisions. In the second paper, there is also an overview in the introduction.
% }

In this paper, we propose a weighted hybrid discontinuous Galerkin (W-HDG) method
inspired by the construction of the complete flux scheme. In the finite volume
context, CFS leads naturally to the Scharfetter--Gummel stabilization with some
additional terms on the right hand side. When the same idea is applied to HDG
methods, one obtains a scheme that resembles exponential fitting methods
\cite{Brezzi1989,brezzi1989numerical,lazarov2012exponential}, with some
important differences: in W-HDG, we discretize directly the primal variables, and thanks to a careful choice of weights used when computing inner products we are able to eliminate from the system the advection term, mimicking what happens when changing to Slotboom variables~\cite{Slotboom1973}.

Similar to HHO and HFV methods, W-HDG methods can be interpreted as a high order alternative to
Scharfetter--Gummel two-point finite volume discretizations. In Section~\ref{sec:problem}
we introduce the problem, and the basic principles behind W-HDG methods. Section
\ref{sec:prelim} presents some general notations common to all HDG
discretizations, while Section~\ref{sec:whdg} discusses stability, continuity,
and convergence properties of the W-HDG method. We present some numerical
examples and draw some conclusions in Sections~\ref{sec:numerical}
and~\ref{sec:conclusion}.

\section{Description of the problem}\label{sec:problem}

Given a function $f\in L^2(\Omega)$  defined on a bounded convex domain
$\Omega\subset\mathbb{R}^d$, boundary data $g_D$ and $g_N$ respectively defined
on $\Gamma_D$ and $\Gamma_N$, where $\Gamma_D \cup \Gamma_N = \Gamma := \partial
    \Omega$, and $\Gamma_D \cap \Gamma_N = \emptyset$, we want to find the pair
$(u,\bj)$ satisfying
\begin{subequations}\label{eq:problem}
    \begin{align}
        \bj  +\alpha \GRAD u - \bbeta u & = 0,\quad    &  & \text{ in } \Omega, \label{eq:flux}  \\
        \DIV \bj                        & = f, \quad   &  & \text{ in } \Omega, \label{eq:div}   \\
        u                               & = g_D, \quad &  & \text{ on } \Gamma_D, \label{eq:dbc} \\
        \bj\cdot \bn                    & = g_N, \quad &  & \text{ on } \Gamma_N . \label{eq:bc}
    \end{align}
\end{subequations}
Here $\bn$ denotes the outward-pointing normal vector of $\partial\Omega$,
$\alpha$ is assumed to be a positive constant, and the vector field $\bbeta$ is
assumed to be in $[L^\infty(\Omega)]^d$.

When $\bj$ is kept as an independent variable as in system~\eqref{eq:problem},
the problem is in its mixed form (see, e.g., \cite{brezzi2012mixed}). By
substituting $\bj$ in equation~\eqref{eq:div} with its definition in
equation~\eqref{eq:flux} one obtains the classical primal formulation of the
drift diffusion problem in divergence form.

If in addition the vector field $\bbeta$ is the gradient of a potential $\psi$,
i.e., $\bbeta = -\nabla \psi$, one can rewrite both the primal and mixed
formulation of the system using the Slotboom change of variable~\cite{Slotboom1973}
\begin{equation}
    \label{eq:slotboom-relation}
    s = u \exp\left(\frac{\psi}{\alpha}\right),
\end{equation}
obtaining from \eqref{eq:problem} the following variant:
\begin{subequations}
    \label{eq:problem-slotboom}
    \begin{align}
        \bj  + \alpha \exp
        \left(-\frac{\psi}{\alpha}\right)\nabla s & = 0,\quad                                         &  & \text{ in } \Omega, \label{eq:flux-slotboom}  \\
        \DIV \bj                                  & = f, \quad                                        &  & \text{ in } \Omega, \label{eq:div-slotboom}   \\
        s                                         & = g_D \exp\left(\frac{\psi}{\alpha}\right), \quad &  & \text{ on } \Gamma_D, \label{eq:dbc-slotboom} \\
        \bj\cdot \bn                              & = g_N, \quad                                      &  & \text{ on } \Gamma_N . \label{eq:bc-slotboom}
    \end{align}
\end{subequations}

In system~\eqref{eq:problem} the fluxes are split in drift and diffusive
contributions, depending, respectively, on $\bbeta$ and $\alpha$. Rewriting
system~\eqref{eq:problem} using the set of variables $(s,\bj)$ in place of $(u,
    \bj)$, has the effect of transforming the problem into a purely diffusive one,
with a highly nonlinear coefficient $\alpha\exp\left(-\frac{\psi}{\alpha}\right)$. 
Due to this fact, we will not use system~\eqref{eq:problem-slotboom} in the discretization (as done by other authors, e.g.,~\cite{brezzi1989numerical}).

The analysis of such problems flourished in the context of semiconductor
devices, where a combination of nonlinear systems with a similar nature of
system~\eqref{eq:problem} describes the behavior of the electrostatic potential
and charge transport. The most famous version of such problems -- the van
Roosbroeck system -- involves three different variables: the electrostatic
potential (playing the role of $\psi$), the electron density, and the hole
density (playing the role of $u$). 

The rewriting of the van Roosbroeck system as a purely diffusive problem using
the Slotboom change of variable ~\eqref{eq:slotboom-relation} enabled the
first analytical result with zero right-hand side in semiconductor theory
\cite{Mock1972}, and was also used more generally in \cite{Markowich1986}. From
the numerical point of view this set of variables was used
to generalize the exponential fitting method to two-dimensional problems \cite{Brezzi1989}.

The Slotboom formulation has mostly analytical advantages, since it allows to
write the system in terms of quasi-linear partial differential equations, allowing 
to apply classical analytical results that can be found for example in
\cite{Taylor2011} which can be used in the context of semiconductors coupled with
circuits in \cite{Al2003} and \cite{Al2010}.

One of the main difficulties in using directly the Slotboom formulation is the
fact that the nonlinear diffusion coefficients remains treatable only when the
ratio $\frac{\psi}{\alpha}$ is of order one. In all other cases it rapidly blows
up or goes to zero, rendering the numerical solution of the system difficult to
deal with (see, e.g., \cite{Brezzi1989}). 

The idea behind our weighted HDG formulation is to introduce a local artificial potential $\psi_K$ such that $\nabla \psi_K = \bbeta$ on each cell, and such that $\alpha\exp(\psi_K/\alpha)$ is locally of order one, to eliminate
the drift term from the system \emph{element-wise}, as in
equation~\eqref{eq:flux-slotboom}.

In order to have a mesh independent solution, we do not use directly the
Slotboom variable $s$, but instead exploit weighted integrals to obtain the same
effect of the Slotboom transformation, and we maintain the primal variable $u$ as our
system unknown.

Such formulation is possible by defining a locally shifted version of $\psi$
which guarantees that $\exp(\frac{\psi}{\alpha})$ is of order one in every element of
the discretization, and using it as a weight in
the integrals that define the local HDG systems (see Section~\ref{sec:whdg} for
the details).

\section{Notation}\label{sec:prelim}
In this section, we introduce some notations used to define our numerical scheme
for the solution of problem \eqref{eq:problem}.

\subsection{Subdivisions} Let us assume that the domain $\Omega\subset\Real^d$
is a polytope. Let $\cT:=\cT(\Omega)$ be a subdivision of $\Omega$ made of
simplices  or quadrilaterals (in two dimensions) or hexahedrons (in three dimensions); we call the elements of $\cT$ ``cells'' and we call the intersection
of two adjacent cells (or the intersection between a cell and $\partial \Omega$)
with positive
$(d-1)$-measure a ``face''. We denote with $\cE^\circ$ and $\cE^\partial$ the collection of
these two kinds of faces and set $\cE := \cE^\circ \cup \cE^\partial$. In what
follows, we assume that $\cT$ is shape-regular in the sense of
\cite{ern2004theory} for simplicial meshes, \ie there exists a positive constant $c_1$ so that for
each cell $K\in \cT$,
\[
    h_K \le c_1 \rho_K,
\]
where $h_K$ is the size of $K$ and $\rho_K$ is the diameter of the largest ball inscribed in $K$, and we refer to~
\cite{ciarlet2002finite} for quadrilateral and hexahedral meshes.
% We further assume that the subdivision $\cT$ is quasi-uniform, \ie there exists a positive constant $c_2$ independent of $\cT$ so that
% \[
%     \max_{K\in \cT} h_K \le c_2 \min_{K \in \cT} h_K .
% \]

\subsection{Approximation spaces and jump operators} For $K\in\cT$, we define
$\cP_k(K)$ and $\mathcal Q_k(K)$ to be the spaces of polynomials defined on $K$ with degree at most
$k$ globally, or in each coordinate direction respectively. The vector spaces $\bcQ(K) = [\cP_k(K)]^d$ or $\bcQ(K) = [\mathcal{Q}_k(K)]^d$ are used to approximate the flux
$\bj$, while the spaces $\cW(K) := \cP_k(K)$ or $\cW(K) := \mathcal{Q}_k(K)$ are used to approximate the scalar solution $u$
in each element $K$ (for simplices and quadrilateral or hexahedral meshes respectively). The corresponding global
finite element spaces $\bcQ(\cT)$ and $\cW(\cT)$ are
respectively
\[
    \bcQ(\cT) := \{\bQ \in [L^2(\Omega)]^d: \bQ|_K \in \bcQ(K)\}
\]
and
\[
    \cW(\cT) := \{W \in L^2(\Omega): W|_K \in \cW(K)\} .
\]
We note that the discontinuous functions $\bQ\in \bcQ(\cT)$ and $w\in \cW(K)$
are double valued when they are restricted to each interior face $e\in
    \cE^\circ$. For two cells $K^+, K^-$ such that $e = K^+\cap K^- \neq \emptyset$, we define the
branches $W_{K^\pm}(\bx):=\lim_{\epsilon\to 0} W(\bx-\epsilon \bn_{K^{\pm}})$
for all $\bx\in e$, where $\bn_{K^\pm}$ denotes the outward normal vectors of
$K^\pm$. The notation for the vector field $\bQ_{K^{\pm}}$ on $e$ is similar. We also
define the jump on $e\in \cE^\circ$,
\[
    \jump{W}_e := W_{K^-}\bn_{K^-} + W_{K^+}\bn_{K^+}
    \quad\text{and}\quad
    \jump{\bQ}_e := \bQ_{K^-}\cdot\bn_{K^-} + \bQ_{K^+}\cdot\bn_{K^+}.
\]
By convention, for $e\in\cE^\partial\cap K \neq \emptyset$ for some cell $K$, we define
$\jump{W}_e = W_{K}\bn_{K}$ and $\jump{\bQ}_e = \bQ_{K}\cdot\bn_{K}$.

For each face $e\in \cE$, we similarly define $\cM(e) = \cP_k(e)$ or $\cM(e) = \mathcal{Q}_k(e)$ the polynomial spaces on $e$ (for simplicial or quadrilateral faces respectively) and
\[
    \cM(\partial K) := \{ \Lambda \in L^2(\cE) : \Lambda|_e \in \cM(e) \text{ for each face} \text{ of } K \}.
\]
The definition of the space $\cM(\cE)$ follows the same idea. Globally, we also define
\begin{equation}\label{eq:space-trace}
    \begin{aligned}
    \cM(G_D;\cE) :=   \{ \Lambda &\in L^2(\cE) :
    \Lambda|_e \in \cM(e) \\
    &\text{ for each } e\in \cE\backslash \Gamma_D
    \text{ and }
    \Lambda|_e = G_D, \text{ for each } e\in \cE\cap \Gamma_D
    \},
    \end{aligned}
\end{equation}
where $G_D\in \cM(\Gamma_D)$ is an approximation of $g_D$ on $\Gamma_D$. In particular, we use the space $\cM(0;\cE)$ for $\cM(G_D;\cE)$ if $G_D = 0$.
% \PF{$\cM_0(\cE)$ could also be read as $D=0$...}\WL{How about using $\cM(0;\cE)$?}
% \subsection{Inner products and norms}
For $K\in\cT$, $e\in \cE_K:=\cE\cap \partial K \neq \emptyset$, and for a positive weight
function $\mu_K$ defined on $K$, we set
\[
    (u,v)_{K,\mu_K}:=\int_K \mu_K uv\diff \bx
    \quad\text{and}\quad
    \langle u,v\rangle_{e ,\mu_K} :=  \int_e \mu_K uv \diff s_\bx ,
\]
and define the inner product $\|.\|_{K,\mu_K} = \sqrt{(.,.)_{K,\mu_K}}$. We also set
\[
    \langle u,v\rangle_{\partial K,\mu_K} := \sum_{e\in \cE_K} \int_e \mu_K uv \diff s_\bx .
\]
Globally, if $\mu$ is defined as a cell-wise function so that $\mu|_{K} = \mu_K$, we introduce the weighted inner-product
\[
    (u,v)_{\cT,\mu} := \sum_{K\in\cT} (u,v)_{K,\mu}.
\]
We define the norm $\|.\|_{\mu} : = \sqrt{(.,.)_{\cT,\mu}}$. If the weight function
$\mu \equiv 1$ in $\Omega$, we omit the subscripts $\mu$ or $\mu_K$ and  we define in a similar way  $\langle.,.\rangle_{\cE}$ and $\langle.,.\rangle_{\Gamma_N}$.

\section{Weighted HDG methods}\label{sec:whdg}

In the section, we introduce our HDG method in the spirit of
\cite{cockburn2009unified} assuming that $\alpha$ is a positive constant, and
that $\bbeta$ is a piecewise constant function subordinate to a subdivision
$\cT$ of $\Omega$. Then there exists a function $\psi \in W^{1,\infty}(\cT)$
such that
\begin{equation}
    \label{eq:definition-v}
    \bbeta = -\GRAD \psi \qquad \text{ locally on each } K.
\end{equation}

Notice that, in general, the function $\psi$ is not continuous across elements of
$\cT$.

\subsection{Discrete formulation of the local problem}
Let us start by considering problem \eqref{eq:problem} in a cell $K\in \cT$. We
multiply equation \eqref{eq:flux} with a weight function
\begin{equation}\label{eq:weight-v}
    \mu_K(\bx) := \exp\bigg(\frac{\psi(\bx)}{\alpha} \bigg)
\end{equation}
satisfying that $\GRAD \mu_K = \tfrac{\GRAD \psi}\alpha \mu_K = -\tfrac{\bbeta}{\alpha}\mu_K$. Whence,
\begin{equation}\label{eq:flux-weighted}
    \mu_K \bj + \alpha \GRAD (\mu_K u) = 0,\quad \text{ in } K
\end{equation}

\begin{remark}\label{rem:weight} We note that the above equation would make
    sense globally (i.e., on the entire $\Omega$) when $\psi\in
        W^{1,\infty}(\Omega)$ and not just on each element separately. This
    assumption is satisfied, for example, by the Poisson equation for the
    electrostatic potential $\psi$ in the van Roosbroeck system. According to
    \cite[Theorem~3.2.1]{Markowich1986}, we can assume that the right hand side of
    the Poisson problem is in $L^p(\Omega)$ if the doping concentration is
    assumed to be an $L^p(\Omega)$ function with $p>d$. Whence by elliptic
    regularity $\psi\in W^{2,p}(\Omega)$ (cf.
    \cite[Theorem~4.3.2.2]{grisvard2011elliptic}) and we guarantee that $\GRAD
        \psi$ is Lipschitz by the Sobolev embedding theorem.
\end{remark}

Next, we test \eqref{eq:flux-weighted} with a vector-valued function $\bQ\in \bcQ(K)$ and multiply \eqref{eq:div} with $\mu_K W$ where $W\in \cW(K)$. Using integration by parts for the term containing $\GRAD (\mu_K u)$ and $\DIV\bj$, our discrete weak formulation becomes: given a local Dirichlet boundary condition $\widehat U\in \cM(\partial K)$ and $f\in L^2(K)$, find $(\bJ,U)\in \bcQ(K)\times \cW(K)$ satisfying that for all $\bQ\in \bcQ(K)$ and $W\in \cW(K)$,
\begin{subequations}\label{eq:discrete-local}
    \begin{align}
        (\bJ, \bQ)_{K,\mu_K} - (\alpha U, \DIV\bQ)_{K,\mu_K}
        + \langle \alpha \widehat U, \bQ\cdot\bn \rangle_{\partial K,\mu_K}                                           & =0,  \label{eq:discrete-flux} \\
        -(\bJ, \GRAD W-\tfrac{\bbeta}\alpha W)_{K,\mu_K} + \langle \widehat{\bJ}\cdot\bn, W\rangle_{\partial K,\mu_K} & = (f,W)_{K,\mu_K} .\label{eq:discrete-div}
    \end{align}
\end{subequations}
Here the numerical flux $\widehat{\bJ}$ is given by
\begin{equation}\label{eq:ldg-flux}
    \widehat\bJ = \bJ + \tau_K (U - \widehat U) \bn,
\end{equation}
with $\tau_K\ge 0$ denoting a piecewise constant function (independent of $h_K$) defined on $\partial K$. We denote the discrete solution with capital letters.

\begin{remark}[Constant $\bbeta$]\label{rem:rescaling}
    In particular, if $\bbeta = -\GRAD \psi$ is a constant vector, we shall use the following weight function
    \begin{equation}\label{eq:weight}
        \mu_K := \exp\bigg(-\frac{\bbeta\cdot(\bx-\bx_K)}{\alpha}\bigg) ,
    \end{equation}
    where $\{\bx_K\}_{K\in\cT}$ is a set of points in $\Rd$. As a consequence,
    the local problem \eqref{eq:discrete-local} is rescaled by the factor
    $\exp(\bbeta\cdot\bx_K/\alpha)$
    % \PF{Not sure I follow... the $\bf{x}$ isn't there anymore on purpose?}\WL{The weight functions \eqref{eq:weight} and $\exp(-\bbeta\cdot\bx/\alpha)$ give the same local problem in \eqref{eq:discrete-local}. I think there is a typo for this factor. The minus sign should be removed.}. 
    We will use the weight function
    \eqref{eq:weight} both for the analysis and for computational purposes by
    choosing different sets of $\{\bx_K\}$; see Section~\ref{subsec:1d} and
    Section~\ref{sec:numerical} for some particular choices.
\end{remark}

The following proposition shows that the local problem admits a unique solution
if the stabilization parameter is strictly positive on at least one face of $K$.

\begin{proposition}[Well-posedness of the local
        problem]\label{prop:local-solvable} The local problem
    \eqref{eq:discrete-local} is uniquely solvable if $\tau_K\ge 0$ and it is
    strictly positive on at least one face of $K$ for simplicial meshes, and it is uniquely solvable if $\tau_K > 0$ on each face of $K$ for quadrilateral or hexahedral meshes.
\end{proposition}

\begin{proof}
    This proof follows closely
    \cite[Proposition~3.2]{cockburn2009unified}. We report it here for completeness. It is sufficient to show that when $\widehat U= 0$ on
    $\partial K$ and $f=0$ in $K$, both $\bJ$ and $U$ vanish. To this
    end, we first apply integration by parts for the first term on the left hand
    side of \eqref{eq:discrete-div} and insert \eqref{eq:ldg-flux}. Whence,
    \begin{equation}\label{eq:discrete-div-2}
        (\DIV \bJ, W)_{K,\mu_K} + \langle \tau_K(U-\widehat U), W \rangle_{\partial K,\mu_K} = 0 .
    \end{equation}
    We choose $\bQ=\bJ/\alpha$ in \eqref{eq:discrete-flux} and $W=U$ in
    \eqref{eq:discrete-div-2} and sum up these two equations to get
    \begin{equation}\label{eq:mid-step-prop-1}
        \frac{1}{\alpha}(\bJ,\bJ)_{K,\mu_K} + \langle \tau_K U, U\rangle_{\partial K,\mu_K} = 0.
    \end{equation}
    Let us consider simplicial meshes first. Since $\alpha>0$ and $\tau_K \ge 0$, we immediately get $\bJ = \bf{0}$ and
    $\sqrt{\tau_K}U = 0$ on $\partial K$, which implies $\tau_K U=0$ on $\partial K$. 
    % \PF{I do not understand the $\Lambda$ subscript...}\WL{$\Lambda$ is removed.} 
    By assumption,
    $\tau_K$ is strictly positive at some face $e$ of $K$, and therefore $U = 0$ on $e$.
    If the polynomial degree $k=0$, we get $U = 0$ and the proof is complete.
    Otherwise, if $k>0$, using the barycentric coordinate system, we can write
    $U = l_e p_{k-1}$ with $p_{k-1}\in \cP_{k-1}(K)$ and where $l_e$ is the
    corresponding barycentric coordinate function that vanishes on $e$. Since we
    have derived that $\bJ = \bf{0}$, according to \eqref{eq:discrete-flux} with
    $\bQ$ satisfying $\DIV \bQ = p_{k-1}$ (this is because the divergence
    operator mapping from $[\cP_k(K)]^d$ to $\cP_{k-1}(K)$ is surjective), we
    have
    \[
        0 = (\alpha U, p_{k-1})_{K,\mu_K} = \int_K \alpha\mu_K l_e p_{k-1}^2 .
    \]
    So $p_{k-1} = 0$ and we conclude that $U = 0$.

    For quadrilateral or hexahedral meshes, we strengthened the assumption on
    $\tau_K$, which is required to be strictly positive on the entire boundary
    $\partial K$ of $K$. From~\eqref{eq:mid-step-prop-1} we conclude that $U=0$
    on $\partial K$, and therefore, integrating by parts
    \eqref{eq:discrete-flux} and choosing $\bQ = \nabla U - \tfrac\bbeta\alpha
    U$ we get $(\alpha \nabla U, \nabla U)_{K, \mu_K} = 0$. Since $U$ is zero on
    $\partial K$, and $\alpha$ is positive, we conclude that $U=0$ on the entire
    $K$.
\end{proof}

\begin{remark}\label{rem:local-solution-operator}
    Proposition~\ref{prop:local-solvable} implies that there exists a linear operator $T$ mapping the data pair $(\widehat U, f)$ to the solution pair $(\bJ, U)$.
\end{remark}

\subsection{Discrete global system}
A global discrete system can be obtained by observing that also at the continuous level, Problem \eqref{eq:problem} can be written in an equivalent form by combining the local systems \eqref{eq:flux-weighted} and \eqref{eq:div} in each cell $K$ and glueing them together by the following transmission condition
\begin{equation}\label{eq:tc}
    \jump{\bj}_e = 0,\qquad \forall e\in \cE^\circ .
\end{equation}
At the discrete level we proceed similarly by glueing together the local systems
in each cell $K$ and imposing the transmission condition \eqref{eq:tc} weakly on
the discrete space. To this end, for each face $e\in\cE^\circ$ and given the
numerical flux $\widehat\bJ$ defined by \eqref{eq:ldg-flux} , we set
\begin{equation}\label{eq:discrete-tc-l2}
    \langle \jump{\widehat\bJ}_e, \xi \rangle_e = 0,\quad\forall \xi \in \cM(e).
\end{equation}

% imposing \eqref{eq:tc}, \eqref{eq:dbc} and \eqref{eq:bc} weakly in the discrete space.
% The boundary conditions follow from \eqref{eq:dbc} and \eqref{eq:bc}. Here we shall similarly form our discrete global system  
% \PF{Why no subscript for jump condition (repeatedly happens in the following)? Also not for inner product?}
We similarly define the discrete counterpart for the boundary conditions: for each $e\in \cE^\partial$ and $\xi \in \cM(e)$, define $\widehat U$ on $\Gamma_D$ and $\widehat \bJ\cdot \bn$ on $\Gamma_N$ by
\begin{equation}\label{eq:discrete-bc-l2}
    \langle \widehat U, \xi \rangle_{e} = \langle g_D,\xi \rangle_e, \text{ if } e\in \Gamma_D
    \quad \text{ and }\quad
    \langle \widehat\bJ\cdot \bn, \xi \rangle_{e} = \langle g_N, \xi \rangle_{e} \text{ if } e\in \Gamma_N ,
\end{equation}
respectively. Based on the above definition, we can write $\widehat U=G_D:=\pi g_D$ with $\pi$ denoting the piecewise $L^2$ projection on $\Gamma_D(\cE)$.
% \PF{Is the sentence complete?}
We also observe that since $\jump{\widehat\bJ} \in \cM(e)$ and the formulation \eqref{eq:discrete-tc-l2} immediately implies that $\jump{\widehat\bJ} = 0$ on $\cE^\circ$. Whence $\widehat \bJ$ is single-valued on $\cE$.

Gathering \eqref{eq:discrete-tc-l2}, \eqref{eq:discrete-bc-l2} as well as the discrete local problem \eqref{eq:discrete-local}, our discrete global formulation reads: find $(\bJ, U, \widehat U)\in \bcQ(\cT)\times \cW(\cT) \times \cM(\pi g_D;\cE)$ such that for all $\bQ\in \bcQ(\cT)$, $W\in \cW(\cT)$ and $\xi \in \cM(0;\cE)$,
\begin{equation}\label{eq:discrete-full}
    \begin{aligned}
        (\bJ, \bQ)_{\cT,\mu} -(\alpha U,\DIV\bQ)_{\cT,\mu}
        + \sum_{K\in\cT} \langle \alpha \widehat U, \bQ\cdot\bn\rangle_{\partial K,\mu}                 & = 0,      \\
        (\DIV\bJ, W)_{\cT,\mu} + \sum_{K\in\cT}\langle \tau_K (U-\widehat U), W\rangle_{\partial K,\mu} & = (f,W)_{\cT,\mu},                                           \\
        \sum_{K\in\cT} \langle \bJ\cdot\bn+\tau_K(U-\widehat U), \xi \rangle_{\partial K\backslash\Gamma_N}               & = \sum_{e\in \Gamma_N}\langle  g_N, \xi \rangle_e  .
    \end{aligned}
\end{equation}

\subsection{Characterization of \texorpdfstring{$\widehat U$}{U}} According to
Remark~\ref{rem:local-solution-operator}, $\bJ$ and $U$ can be written as an
operator acting on $\widehat U$ and $f$.  So when solving the discrete system
\eqref{eq:discrete-full}, we can first eliminate the unknowns $\bJ$ and $U$ to
rewrite the system only for $\widehat U$. In this subsection, we shall derive such
system for $\widehat U$ in order to investigate the well-posedness of
\eqref{eq:discrete-full}. 
To start with, when $f$ or $g_D$ vanishes, we denote the unique triplet
\begin{equation}\label{eq:notation-sol}
    (\bJ_0, U_0, \widehat \bJ_0)\quad \text{ or }\quad(\bJ_f, U_f, \widehat \bJ_f)
\end{equation}
that solves the local problem \eqref{eq:discrete-local} for all $K\in\cT$ thanks to Remark~\ref{rem:local-solution-operator}.
% \PF{There were no triplets in the remark... Also no subscripts...}
. Whence, in view of the discrete
transmission condition \eqref{eq:discrete-tc-l2} as well as the Neumann boundary
condition defined in \eqref{eq:discrete-bc-l2}, the solution $\widehat
    U\in \cM(\pi g_D;\cE)$ satisfies that
\begin{equation}\label{eq:global-system}
    a(\widehat U, \xi) = b(\xi),\quad\forall \xi\in \cM(0;\cE),
\end{equation}
where
\[
    \begin{aligned}
        a(\widehat U, \xi) & = -\langle \jump{\widehat \bJ_0}, \xi \rangle_{\cE}                                     \\
        b(\xi)             & = \langle \jump{\widehat \bJ_f}, \xi \rangle_{\cE} - \langle g_N, \xi\rangle_{\Gamma_N}.
    \end{aligned}
\]
Here $\jump{.}$ denotes the jump the flux on $\cE$.
% \PF{Notation for inner product was previously with edge as subscript, not set of edges or $\Gamma_N$... Also (15) uses $\langle \widehat\bJ\cdot \bn, \xi \rangle_{e}$ not the transmission condition? I'm guessing this is intended but should we not introduce this notation or say that we go from local to global?}

\subsection{A decomposition of the bilinear form \texorpdfstring{$a(.,.)$}{a}}
Let $\mu$ be a positive function defined on $\Omega$ so that $\mu|_K = \mu_K$ for every $K\in \cT$. Then we decompose the bilinear form $a(.,.)$ by
\begin{equation}\label{eq:decomp}
    a(\widehat U, \xi) = \widetilde a(\widehat U,\xi)
    + E  (\widehat U, \xi),
\end{equation}
where
\[
    \widetilde a(\widehat U,\xi) =  -\langle \jump{\mu \widehat \bJ_0}, \xi \rangle_{\cE}
    \quad\text{and}\quad
    E(\widehat U, \xi) =  -\big(\langle \jump{ \widehat \bJ_0}, \xi \rangle_{\cE} -\langle \jump{\mu \widehat \bJ_0}, \xi \rangle_{\cE} \big) .
\]
% \PF{Is there maybe a sign error w.r.t. the $+$?}
The following lemma provides a characterization of the bilinear form
$\widetilde a(.,.)$ assuming zero Dirichlet boundary conditions.

\begin{lemma}[Characterization of $\widetilde a$]\label{lem:symmetric-part} The
    bilinear form $\widetilde a(.,.)$ is symmetric semidefinite on
    $\cM(0;\cE)\times\cM(0;\cE)$. Moreover, given any $\Lambda\in \cM(0;\cE)$, we denote
    $(\bJ_\Lambda, U_\Lambda)$ the corresponding cell solutions by solving the
    local problem \eqref{eq:discrete-local} with $f=0$, $\widehat U = \Lambda$ and recall \eqref{eq:notation-sol}.
    Then there holds that for $\xi \in \cM(0;\cE)$,
    \begin{equation}\label{eq:characterization}
        \widetilde a(\Lambda,\xi) = \frac{1}\alpha (\bJ_\xi,\bJ_\Lambda)_{\cT,\mu}
        + \langle \jump{\mu\tau (U_\Lambda - \Lambda)(U_\xi-\xi)}, 1\rangle_{\cE}
    \end{equation}
    where $\tau$ is a function defined on $\cE$ so that $\tau|_{\partial K} = \tau_K$.
\end{lemma}
\begin{proof}
    Before deriving \eqref{eq:characterization}, we first provide some essential identities. By summing up the local flux problem \eqref{eq:discrete-flux} with $\widehat U = \Lambda$ 
    % \PF{$\lambda$?} 
    for all $K\in \cT$, we have
    \begin{equation}\label{eq:relation-lambda-1}
        \frac1\alpha(\bJ_\Lambda,\bQ)_{\cT,\mu} - ( U_\Lambda,\DIV \bQ)_{\cT,\mu}
        =- \langle \Lambda, \jump{\mu \bQ}\rangle_{\cE}
    \end{equation}
    We use integration by parts for the first term on the left hand side of
    \eqref{eq:discrete-div} and sum up the equations for all $K\in \cT$ to
    obtain that (note that $f=0$)
    \begin{equation}\label{eq:relation-lambda-2}
        (\DIV \bJ_\Lambda, W)_{\mathcal T,\mu} = \langle \jump{\mu(\bJ_\Lambda-\widehat\bJ_\Lambda)W}, 1\rangle_{\cE} .
    \end{equation}
    Using the above relations, we arrive at
    \[
        \begin{aligned}
            \widetilde a(\Lambda, \xi)             & = -\langle \jump{\mu\bJ_\Lambda},\xi\rangle_{\cE} -\langle \jump{\mu(\widehat{\bJ}_\Lambda-\bJ_\Lambda)},\xi\rangle_{\cE} &                                                        & \\
                                                   & = \frac{1}\alpha (\bJ_\xi,\bJ_\Lambda)_{\cT,\mu}
            -(U_{\xi},\DIV\bJ_{\Lambda})_{\cT,\mu} &                                                                                                                           &                                                          \\
                                                   & \qquad -\langle \jump{\mu(\widehat{\bJ}_\Lambda-\bJ_\Lambda)},\xi\rangle_{\cE}                                            &                                                        &
            \text{by \eqref{eq:relation-lambda-1} with } \Lambda=\xi, \bQ = \bJ_{\Lambda}                                                                                                                                                 \\
                                                   & = \frac{1}\alpha (\bJ_\xi,\bJ_\Lambda)_{\cT,\mu}
            + \langle \jump{\mu(\widehat\bJ_\Lambda- \bJ_\Lambda)(U_\xi-\xi)}, 1\rangle_{\cE}
                                                   &                                                                                                                           & \text{by \eqref{eq:relation-lambda-2} with } W=U_\xi .
        \end{aligned}
    \]
    The proof in complete by inserting the definition of the numerical flux \eqref{eq:ldg-flux}.
\end{proof}
%\PF{$\bJ_\xi$ does not appear in (21), yet after the second equality sign it appears here...}

\begin{remark}
    The above theorem also shows that the local problems
    \eqref{eq:discrete-local} for all $K\in \cT$ together with the weak
    transmission condition
    \[
        \langle \jump{\mu\widehat\bJ},\xi \rangle_\cE = 0,\quad\forall \xi\in \cM(0;\cE) .
    \]
    give rise to a symmetric discrete global system. %\PF{symmetric what?}
\end{remark}

\begin{remark}\label{rem:local-nonnegativity}
    For $\Lambda\in \cM(0;\cE)$, we recall that
    \[
        0\le  \widetilde a(\Lambda,\Lambda) = \sum_{K\in\cT} -\langle \widehat \bJ_\Lambda\cdot \bn, \Lambda \rangle_{\partial K,\mu_K} =: \sum_{K\in \cT} I_K .
    \]
    % \PF{Are we maybe missing some normal vector in the second to last term?}
    We can also show the above by proving the nonnegativity of $I_K$. Indeed,
    let $\widehat U=\Lambda$, and denote with $\bJ_\Lambda$ and $U_\Lambda$ the
    corresponding solutions to the local problems in \eqref{eq:discrete-local}.
    Then choosing $\bQ=\bJ_\Lambda$, and $W=U_\Lambda$ in
    \eqref{eq:discrete-local}, and summing up the equations with $f=0$ we obtain
    \[
        I_K = \frac1\alpha(\bJ_\Lambda,\bJ_\Lambda)_{K,\mu_K}
        + \langle \tau_K(U_\Lambda-\Lambda),U_\Lambda-\Lambda\rangle_{\partial K,\mu_K} \ge 0.
    \]
\end{remark}

\subsection{Well-posedness of the global problem}
We recall that the global weight function $\mu$ satisfies that $\mu|_K = \mu_K$
for $K\in\cT$ with $\mu_K$ defined by \eqref{eq:weight}. In general, $\mu$ is
discontinuous across elements $K$, due to the discontinuity of $\bbeta$ and to
the choice of a different $\bx_K$ in each element $K$. However, if the vector
field $\bbeta$ is the gradient of a globally continuous, piecewise linear
$\psi$, then we can choose $\bx_K = \boldsymbol 0$ for every element $K$ and
obtain a resulting $\mu$ which is continuous on the entire $\Omega$ (owing to
the definition \eqref{eq:weight-v}). With this property it is straight forward to
show that \eqref{eq:global-system} is well-posed under some conditions on the
choice of the stabilization parameter $\tau_K$ and on the properties of the
vector $\bbeta$.

\begin{theorem}\label{thm:solvable-continuous} Assume that i) $\psi$ is
    continuous and piecewise linear, and that $\bbeta = -\GRAD \psi$ is
    piecewise constant on $\mathcal T$, ii) $\tau_K$ is strictly positive on
    $\partial K$ for all $K\in \cT$. Then problem \eqref{eq:global-system} is
    uniquely solvable.
\end{theorem}
\begin{proof}
    Since the system matrix is square, it suffices to show that given vanishing
    $f$, $g_D$, $g_N$ such that $a(\Lambda,\xi) = 0$ for all $\xi\in
    \cM(0;\cE)$, we can only obtain trivial solutions on $\cE$.
    Since $\mu$ is continuous when we choose one $\bx_K=\boldsymbol 0$ for all
    $K\in\cT$ and $\widehat\bJ_\Lambda$ is single-valued thanks to
    \eqref{eq:decomp},  we have $E(\Lambda,\xi) = 0$ and hence
    $0=a(\Lambda,\xi)=\widetilde a(\Lambda, \xi)$. Invoking
    Lemma~\ref{lem:symmetric-part} as well as
    Remark~\ref{rem:local-nonnegativity}, we immediately obtain $\bJ_\Lambda=0$
    in $K$ and $\tau_K(U_\Lambda-\Lambda)=0$ on $\partial K$. By assumption for
    $\tau_K$, $\Lambda=U_\Lambda$ on $\partial K$. We next want to show that
    $U_\Lambda = \mathbf{0}$. Applying integration by parts in
    \eqref{eq:discrete-flux} implies that
    \begin{equation}\label{eq:mid-step-2}
        (\GRAD(\mu U_\Lambda), \bQ)_{K}
        -\langle \mu(U_\Lambda-\Lambda), \bQ\cdot\bn\rangle_{\partial K} = 0 .
    \end{equation}
    We again apply $U_\Lambda=\Lambda$ on $\partial K$ into \eqref{eq:mid-step-2} to get
    \begin{equation}\label{eq:mid-step-3}
        (\GRAD(\mu U_\Lambda),\bQ)_{K} = 0, \quad\forall \bQ\in \bcQ(K) .
    \end{equation}
    Letting $\bQ = \GRAD U_\Lambda - \tfrac\bbeta\alpha U_\Lambda \in \bcQ(K)$ implies that
    \[
        0 = (\GRAD(\mu U_\Lambda),\bQ)_{K} = (\GRAD(\mu U_\Lambda),\GRAD(\mu U_\Lambda))_{K,1/\mu} .
    \]
    Hence, we deduce that $U_\Lambda = C_K/\mu$ for some constant $C_K$ in each $K$ and $\Lambda=U_\Lambda$ on $\partial K$. If $\bbeta=\mathbf{0}$, we derive that $U_\Lambda$ is a constant and hence $U_\Lambda=0$ due to the vanishing Dirichlet boundary condition. If $\bbeta\neq \mathbf{0}$, as $U_\Lambda\in \cW(K)$, we conclude that $U_\Lambda=0$ and hence $\Lambda=0$. The proof is complete.
\end{proof}
\begin{remark}
Notice that assumption ii) above is more restrictive than the one required in
Proposition~\ref{prop:local-solvable} for simplicial meshes, where strict
positivity of $\tau_k$ is required only on part of $\partial K$, instead of on all
$\partial K$. This is also more restrictive w.r.t. classical results for HDG
methods, and derives from the fact that the weight function limits the
applicability of the technique used in ~\cite[Lemma 3.1]{cockburn2009unified}.
\end{remark}

% \begin{remark}
%     We can extend the above argument to show the well-posedness of the global system for the case when $\bbeta = -\GRAD \psi$ with $\psi\in C^0(\Omega)\cap \cP_1(\cT)$. So $\bbeta$ is a piecewise constant function and $\mu$ is continuous. This will be used to develop the numerical scheme for drift-diffusion system when the electric potential $\psi$ is approximated by continuous piecewise linear functions subordinate to $\cT$.
% \end{remark}

\subsection{Well-posedness for the one-dimensional case with \texorpdfstring{$\beta\in L^\infty(\Omega)$}{beta}}\label{subsec:1d}

If $\Omega=[a,b]\subset \Real$, we can still assume that $\beta$ is a piecewise
constant function even though $\psi$ may not be continuous.

\begin{corollary}\label{cor:solvable-1d} Assume that $\beta$ is a piecewise
    constant function with $\beta_K:=\beta|_K $. We also assume that $\tau_K$ is
    positive on $\cE$. Then we can construct $\mu$ such that problem
    \eqref{eq:global-system} is uniquely solvable.
\end{corollary}
\begin{proof}
    Our goal is to construct a continuous weight function $\mu$ so that
    $E(U_\Lambda,\xi)=0$ in the proof of Theorem~\ref{thm:solvable-continuous}.
    We set $\{x_i\}_{i=0}^N$ to be the grid points so that
    $a=x_0<\ldots<x_N=b$ and set $K_i=[x_{i-1},x_i]$. For each
    $i=1,\ldots,N-1$, we define $\mu_{K_i}$ and $\mu_{K_{i+1}}$ so that $\mu$ is
    continuous at $x_i$. To achieve this, we set
    \begin{equation}\label{eq:mu-continuous}
        \frac{\beta_{K_i}}{\alpha}(x_i-x_{K_i}) = \frac{\beta_{K_{i+1}}}{\alpha} (x_i-x_{K_{i+1}}) .
    \end{equation}
    Given a fixed $x_{K_{1}}\in \Real$, we can solve for $x_{K_i}$ for $i=1,\ldots, N$ based on \eqref{eq:mu-continuous} and hence the resulting $\mu$ is continuous. The proof is complete according to the proof of Theorem~\ref{thm:solvable-continuous}.
\end{proof}

\subsection{Well-posedness for a general case}\label{subsec:general} In this
section we want to extend the assumption on $\psi$, namely we assume that there
exists a function $\psi\in W^{1,\infty}(\Omega)$ so that $\bbeta = -\GRAD
\psi\in [L^\infty(\Omega)]^d$. We then choose the weight function
\eqref{eq:weight-v}. To show the well-posedness of the global problem
\eqref{eq:global-system}, we shall use the Poincar\'e inequality: for $w\in
H^1_D(\Omega)$, there exists a constant $C_p\sim C\diam(\Omega)$ satisfying
\[
    \|w\|_{L^2(\Omega)} \le C_p \|\GRAD w\|_{L^2(\Omega)} .
\]

\begin{theorem}\label{thm:solvable-general}
    Suppose that $\psi\in W^{1,\infty}(\Omega)$. We further assume that $\|\bbeta\|_{L^\infty(\Omega)}$ is small enough such that the Poincar\'e constant satisfies
    \begin{equation}\label{ineq:cp}
       \frac{\|\bbeta\|_{L^\infty(\Omega)} C_p}{2\alpha} < 1.
    \end{equation}
    Under the assumption that $\tau_K$ is positive on $\partial K$ for all $K\in\cT$, the discrete problem \eqref{eq:global-system} is then uniquely solvable.
\end{theorem}
\begin{proof}
    Following the proof of Theorem~\ref{thm:solvable-continuous}, let us start our proof at \eqref{eq:mid-step-3}. Under the current setting, $\bQ=\GRAD U_\Lambda-\tfrac\bbeta\alpha U_\Lambda\notin \bcQ(K)$. Here we choose $\bQ=\GRAD U_\Lambda$. We rewrite \eqref{eq:mid-step-3} to get
    \[
        0 = (\GRAD(\mu U_\Lambda), \GRAD U)_K
        = \|\GRAD(\mu^{1/2}U_\Lambda)\|_{L^2(K)}^2 - \left\|\frac{|\bbeta|}{2\alpha}\mu^{1/2}U_\Lambda\right\|_{L^2(K)}^2,
    \]
    where we indicate with $|\bbeta|$ the Euclidean norm of $\bbeta$.
    Now we sum up the above equation for all $K\in \cT$. Note that $U_\Lambda=\Lambda$ implies that $U_\Lambda\in H^1_D(\Omega)$. So $\mu^{1/2}U_\Lambda\in H^1_D(\Omega)$. By the Poincar\'e inequality, we have
    \[
        \begin{aligned}
            0 & = \|\GRAD(\mu^{1/2}U_\Lambda)\|_{L^2(\Omega)}^2 -  \left\|\frac{|\bbeta|}{2\alpha}\mu^{1/2}U_\Lambda\right\|_{L^2(\Omega)}^2  \\
              & \ge\left( \frac{1}{C_p^2} - \frac{\|\bbeta\|^2_{L^\infty(\Omega)}}{4\alpha^2}\right) \|\mu^{1/2}U_\Lambda\|_{L^2(\Omega)}^2 .
        \end{aligned}
    \]
    Assumption  \eqref{ineq:cp} shows that the constant on the right hand side above is strictly positive. Hence $U_\Lambda = 0$ and the proof is complete.
\end{proof}

Theorem~\ref{thm:solvable-general} shows that, under the same conditions of well-posedness for
classical drift-diffusion systems, we get well posedness also of our weighted formulation.

\subsection{A limiting case in the one-dimensional setting}\label{ssec:limiting-case}

Let us consider the numerical scheme \eqref{eq:discrete-full} in 1d with $k=0$, \ie using piecewise constant functions. Set $\Omega = (a,b)$ for some real numbers $a<b$ and $\Omega$ can be subdivided with $a=x_0 <\ldots < x_{N} = b$. For $i=0,\ldots, N-1$, denote $h_i=x_{i+1}-x_i$. We also denote $\{\Lambda_i\}_{i=0}^N$ the trace values on $x_i$ and $\{(J_i,U_i)\}_{i=1}^N$ the cell values on the interval $(x_{i-1},x_i)$.

Using the constant test functions in \eqref{eq:discrete-full}, we can write $J_i$ as the function of $\{\Lambda_i\}$. That is for $i=1,\ldots,N$,
\begin{equation}\label{eq:weighted-sg}
    J_i = \frac{\alpha}{h_i} \bigg(\ber{\frac{\beta h_i}{\alpha}}\Lambda_{i-1} - \ber{-\frac{\beta h_i}\alpha } \Lambda_{i} \bigg) ,
\end{equation}
where $\ber{.}$ is the Bernoulli function:
\[
    \ber t := \frac{t}{e^t - 1} .
\]
The above flux is usually referred to as the Scharfetter-Gummel flux. For the second equation in \eqref{eq:discrete-full}, we write for $i=1,\ldots,N$,
\begin{equation}\label{eq:weighted-u}
    U_i = \frac{e_i \Lambda_{i-1} + e_{-i}\Lambda_i}{e_h+e_{-h}} + \frac{F_i}{\tau_i(e_{i}+e_{-i})}
\end{equation}
with $e_{\pm i}:= \exp(\pm\tfrac{\beta h_i}{2\alpha})$ and $F_i = \int_{x_{i-1}}^{x_i}f\mu$. In terms of the discrete transmission condition, here we simply assume that the stabilization parameter $\tau$ is a positive constant on $\cE$. Under the piecewise constant setting, we write for $i = 1,\ldots, N-1$,
\begin{equation}\label{eq:weighted-tc}
    J_{i} - J_{i+1} + \tau (  U_i + U_{i+1} -2\Lambda_i ) = 0 .
\end{equation}
Inserting \eqref{eq:weighted-sg} and \eqref{eq:weighted-u} into \eqref{eq:weighted-tc} and letting $\tau$ tend to zero to write the global problem for $\Lambda$ by
\begin{equation}\label{eq:sg}
    \begin{aligned}
         & \frac{\alpha}{h_{i}} \ber{\frac{\beta h_{i}}{\alpha}} \Lambda_{i-1}          \\
         & \qquad
        -\bigg( \frac{\alpha}{h_{i}} \ber{-\frac{\beta h_{i}}{\alpha}}
        + \frac{\alpha }{h_{i+1}} \ber{\frac{\beta h_{i+1}}{\alpha}}  \bigg)  \Lambda_i \\
         & \qquad\qquad
        +\frac{\alpha}{h_{i+1}} \ber{-\frac{\beta h_{i+1}}{\alpha}}\Lambda_{i+1} = F_i .
    \end{aligned}
\end{equation}
The above system is identical to the system built using the finite volume method
together with the Scharfetter-Gummel flux, justifying the interpretation of the
W-HDG method as a high order generalization of the Scharfetter-Gummel stabilization.

\section{Postprocessing}

% Motivation: In HDG you can get better convergence order
% via postprocessing. A priori not clear whether this is also true for weighted
% HDG.

Similarly to what happens in classical HDG
methods~\cite{cockburn2009hybridizable}, we can improve the convergence order of
the W-HDG method by postprocessing the solution. We propose two different
postprocessing procedures in an element-by-element fashion to obtain the density
approximations of $u$ using $\cP_{k+1}$ or $\mathcal Q_{k+1}$ elements for
simplicial or quadrilateral/hexahedral meshes respectively so that these
approximations converge with order of $O(h^{k+2})$ in both $L^2(\Omega)$ and
$L^\infty(\Omega)$ norms.

\subsection{Postprocessing by \texorpdfstring{$L^2(K)$}{l22} minimization} We
consider a density approximation $U_*\in \cP_{k+1}(\cT)$ or $\mathcal
Q_{k+1}(\cT)$ so that for each $K\in \cT$, $U_*$ minimizes the functional
\[
    \int_K |\alpha \GRAD U_* -\bbeta U + \bJ|^2 \diff \bx
\]
under the constraint that $\int_K (U_*-U)\diff \bx = 0$. We note that the constraint is necessary since $U_*+C$ is also a minimizer of the above functional with $C$ being any constant. The approximation $U_*$ can be obtained by solving the following local problem:
\begin{equation}\label{eq:l2-minimization}
    \begin{aligned}
        (1, U_*)_K                & = (1, U)_K,                &                             \\
        (\alpha\GRAD U_*,\GRAD W)_K & = (\bbeta U-\bJ, \GRAD W)_K, & \forall W\in \cP_{k+1}(K) .
    \end{aligned}
\end{equation}
We note that $U_*$ has a superconvergence property with rate $O(h^{k+2})$ in the
$L^2(\Omega)$-sense since both $U$ and $\bJ$ converge in optimal rate
$O(h^{k+1})$ and the local average of $U$, \ie $\tfrac{1}{|K|}\int_K U$ super
converges at rate $O(h^{k+2})$ (cf. \cite{cockburn2009superconvergent}). The
numerical simulation illustrated in Section~\ref{sec:numerical} also suggests
that if $u$ is smooth enough, the postprocessed approximation also converges in
$L^\infty(\Omega)$ with rate $O(h^{k+2})$. We remark, however, that this
postprocessing results does not satisfy the local problem
\eqref{eq:discrete-local} in any sense.

\subsection{Postprocessing based on the local problem \texorpdfstring{\eqref{eq:discrete-local}}{}}
We next construct a postprocessing approximation based on the local problem \eqref{eq:discrete-local} following an idea from \cite{nguyen2009implicit}. The procedure needs more computations compared to the $L^2(K)$-minimization.

\subsubsection{Reconstruction for the flux}
We shall first reconstruct the flux approximation $\bJ$ in the space
$H(\text{div},\Omega)$, the vector space in $[L^2(\Omega)]^d$ so that its
divergence belongs to $L^2(\Omega)$. To achieve this, we consider discontinuous
piecewise Raviart-Thomas-N\'ed\'elec space $\RT_k(\cT)=\{\bQ\in [L^2(\Omega)]^d
: \bQ|_{K}\in \RT_k (K)\}$, where $\RT_k(K) := \text{Im}(\DIV \cP_{k+1}(K)^d)$
or $\RT_k(K) := \text{Im}(\DIV \mathcal{Q}_{k+1}(K)^d)$ for simplicial or
quadrilateral/hexahedral meshes respectively. Using the cell approximation $\bJ$
as well as the trace approximation $\widehat \bJ$, we define the reconstruction
$\bJ_{div}\in \RT_k(\cT)$ as
\begin{equation}\label{eq:rtn-proj}
    \begin{aligned}
        \langle \bJ_{div}\cdot\bn, \bQ \rangle_{e}
                           & =\langle \widehat{\bJ}\cdot\bn, \xi \rangle_{\partial K}, &  & \forall \xi\in \cM(e),                            \\
        (\bJ_{div}, \bQ)_K & = (\bJ, \bQ)_K,                         &  & \forall \bQ\in [\cM_{k-1}(K)]^d\text{ if } k\ge 1 .
    \end{aligned}
\end{equation}
Here we have to point out some properties of $\bJ_{div}$: i) since
$\widehat\bJ$ is single valued on $\cE$, the normal component of $\bJ_{div}$ is
continuous across interior faces. So $\bJ_{div}\in H(\text{div},\Omega)$; ii)
\cite{cockburn2009hybridizable} shows that both $\bJ_{div}$ and $\DIV\bJ_{div}$
converge in optimal rate $O(h^{k+1})$, which is crucial to the postprocessing of
$U$ in the next step.

\subsubsection{Postprocessing of \texorpdfstring{$U$}{}}
Our reconstruction of $U$ is based on the problem \eqref{eq:problem} defined on each element $K$ together with a Neumann boundary condition, \ie at the continuous level, we want to find the pair $(\bj^*, u^*)$ satisfying
\[
    \begin{aligned}
        \bj^* + \alpha \GRAD u^* -\bbeta u^* & = \boldsymbol 0,                 &  & \text{ in } K,         \\
        \DIV \bj^*                           & = \DIV\bJ_{div},     &  & \text{ in } K,         \\
        \bj^*\cdot \bn                       & = \bJ_{div}\cdot\bn, &  & \text{ on }\partial K.
    \end{aligned}
\]
Similar to the $L^2(K)$ minimization scheme, we also need the constraint $\int_K (u^*-U) = 0$ to guarantee that the local reconstruction is unique.

Now we are ready to define an approximation of $(\bj^*, u^*)$ by modifying the discrete local problem \eqref{eq:discrete-local}: letting $\bcQ^*(K)$, $W^*(K)$ and $\cM^*(\partial K)$ be the polynomial spaces like $\bcQ(K)$, $W(K)$ and $\cM(\partial K)$ but with the degree $k+1$, we want to find $(\bJ^*, U^*, \widehat U^*)\in \bcQ^*(K)\times W^*(K)\times \cM^*(\partial K)$ such that $\int_K (U^*-U) = 0$ and for all $\bQ\in \bcQ^*(K)$, $W\in \cW^*(K)$ and $\xi\in \cM^*(K)$,
\begin{equation}\label{eq:local-postprocessing}
    \begin{aligned}
        (\bJ^*, \bQ)_{K,\mu_K} - (\alpha U^*, \DIV\bQ)_{K,\mu_K}
        + \langle \alpha \widehat U^*, \bQ\cdot\bn \rangle_{\partial K,\mu_K}                                             & =0,                             \\
        -(\bJ^*, \GRAD W-\tfrac{\bbeta}\alpha W)_{K,\mu_K} + \langle \widehat{\bJ}^*\cdot\bn, W\rangle_{\partial K,\mu_K} & = (\DIV\bJ_{div},W)_{K,\mu_K},                            \\
        \langle \widehat{\bJ}^*\cdot \bn,\xi \rangle_{\partial K}                                                         & =  \langle \bJ_{div}\cdot \bn,\xi \rangle_{\partial K} ,
    \end{aligned}
\end{equation}
with the numerical flux $\widehat{\bJ}^*$ given by
\[
    \widehat\bJ^* = \bJ^* + \tau_K (U^* - \widehat U^*) \bn .
\]

Proposition~\ref{prop:local-solvable} shows that the above problem admits a
unique solution. Thanks to the $O(h^{k+1})$ rate of convergence for both
$\bJ_{div}$ and $\DIV \bJ_{div}$, and to the superconvergence property for the
local average of $U$, in
\cite{cockburn2008superconvergent,cockburn2009superconvergent} the authors show
that, for $k\ge 1$, the postprocessing approximation $U^*$ converges with rate
$O(h^{k+2})$ in the $L^2(\Omega)$ sense. This convergence behavior is confirmed
numerically also for the W-HDG method in Section~\ref{sec:numerical}.

\section{Numerical illustration}\label{sec:numerical} 

In this section, we will present some two-dimensional numerical experiments to
verify the stability and convergence of our proposed numerical scheme. Our
numerical implementation relies on the \texttt{deal.II} finite element library
\cite{dealII93,dealIIdesign} using bi-polynomial elements $\cQ_k$ under a
sequence of quadrilateral subdivisions. 

Before presenting the examples, we introduce a Gaussian quadrature rule to
compute the inner-products $(v,w)_{K,\mu_K}$ and $\langle v,w\rangle_{\partial
K,\mu_K}$ exactly when the element $K=\Pi_{i=1}^d(a_i,b_i)$ is a hyperrectangle.
The finite element spaces are defined using bi-polynomials, and $\bbeta$ is a
piecewise constant function. The weight function $\mu$ is defined as in
\eqref{eq:weight} with $\bx_K$ denoting the center of $K\in \mathcal T$.

\subsection{Quadrature schemes on rectangular elements with constant
\texorpdfstring{$\bbeta$}{}} We first consider the one-dimensional case
$(v,w)_{\widehat K,\mu}$ where $\widehat K = (0,1)$ and $\mu(\hat x) =
\exp(-\hat\beta (\hat x-\hat x_{\widehat K}))$ for some constant $\hat\beta$. We
generate a Gaussian quadrature scheme $\widehat Q(\widehat\beta,\widehat K)$
based on orthogonal polynomials up to degree $k_0$ (i.e., Gauss quadrature rules
with $k_0$ quadrature points, $k_0\geq 1$) with respect to the inner-product
$(v,w)_{\widehat K,\mu}$. Our integral computation is exact when $vw$ is a
polynomial with degree at most $2k_0-1$. For any interval $K=(a,b)$, we can
compute the inner-product by the change of variable $x=a+h_K\hat x$ with
$h_K=b-a$. Whence,
\[
    \int_K e^{-\beta (x-x_K)} v(x)w(x) \diff x = h_K\int_{\widehat  K} e^{-\beta h_K(\hat x - \hat x_{\widehat K})} \hat u \hat v \diff \hat x .
\]
The integral on the right hand side is computed exactly by the quadrature
$\widehat Q(h_K\beta,\widehat K)$. For the multidimensional case with a constant
$\bbeta$, we apply the tensor product quadrature rule to compute $(v,
    w)_{K,\mu}$ based on the 1d quadrature scheme. Denoting the corresponding
quadrature by $Q_d(\bbeta, K)$ and given a face $e$ of the element $K$, we note
that we can compute similarly $(v,w)_{e, \mu}$ by $Q_{d-1}(\bbeta,
    e)\exp(-\bbeta\cdot(\bx_e-\bx_K))$, with $\bx_e$ denoting the center of the
face.

\subsection{2D convergence tests on uniform Cartesian grids}
We test our proposed numerical scheme \eqref{eq:discrete-full} to approximate the problem \eqref{eq:problem} on the unit square $\Omega=(0,1)^2$. Given a constant vector $\bbeta=(\beta_1,\beta_2)^{\mathtt T}$, we define
the analytic solution to be
\[
    u(x_1,x_2) = x_1 x_2 \frac{(1-\exp((x_1-1)\beta_1))(1-\exp((x_2-1)\beta_2))}{(1-\exp(-\beta_1))(1-\exp(-\beta_2))}
\]
so that the solution vanishes at the boundary (as in the 1d case) and produces boundary layers along $x_1=1$ and $x_2=1$ when $\beta_1$ and $\beta_2$ are large enough, respectively.

\subsubsection{Convergence for  \texorpdfstring{$(\bJ, U)$}{}} Here we set
$\{\mathcal T_j\}_{j=1}^6$ to be sequence of uniform Cartesian grids with the
mesh size $h_j = 2^{-j-1}\sqrt{2}$. In Table~\ref{table:uniform-convergence}, we
report the $L^2(\Omega)$-errors between $(\bj, u)$ and $(\bJ, U)$ as well as the
$L^\infty(\Omega)$ errors for $U$ using different degrees of bi-polynomials. We
note that the stabilization parameter $\tau_K$ is fixed to be 1 on all edges.
The convergence rate showing in the table is computed with
\begin{equation}
    \text{rate}_{j+1} := \log(e_{j+1}/e_j) /\log (\#\text{DoF}_{j+1}/\# \text{DoF}_{j}),
\end{equation}
where $\#\text{DoF}_j$ refers to the number of degrees of freedoms for mesh
refinement at level $j$ and $e_j$ represents the measured error on the
corresponding refinement level. From Table~\ref{table:uniform-convergence}, we
observe the optimal rate of convergence $O(h^{k+1})$ in both $L^2(\Omega)$ and
$L^\infty(\Omega)$ norms. Here we note that the $L^\infty(\Omega)$-error is
approximated by comparing the values on the 6th-order Gaussian quadrature points
on each cell. We also report the error behavior of the largest cell-wise average
of $U$. That is $\max_{K\in\cT}(u-U,1)_K$. 
% \PF{Not sure I understand these three words here. Is that the cell-wise average?} 
We observe that this quantity superconverges with rate $O(h^{2k+1})$ when $k\ge
1$. This helps us to obtain higher order rate of convergence in the procedure of
the postprocessing using bi-polynomials $\mathcal Q_{k+1}$.
\begin{comment}
\WLc{The
    quantity actually converges in $2k+1$ order. This seems strange to me since in
    deal.II step-51, it says that this will converge in the order $k+2$. I think I
    should first double check my code. If the code seems to be okay, I should find a
    reference which contains a statement of this.}
\end{comment}

\begin{table}[hbt!]
    \begin{center}
        \resizebox{\textwidth}{!}{
        \begin{tabular}{|r|r|r|c|c|c|c|c|c|c|c|} \hline
            deg                                                    & \# cells & \# DoFs &
            \multicolumn{2}{|c|}{$\| \bj - \bJ \|_{L^2(\Omega)}$}  &
            \multicolumn{2}{|c|}{$\| u - U \|_{L^2(\Omega)}$}      &
            \multicolumn{2}{|c|}{$\| u - U \|_{L^\infty(\Omega)}$} &
            \multicolumn{2}{|c|}{avg}                                                                                                                                  \\ \hline
                                                                   &          &         & error     & rate & error     & rate  & error     & rate  & error     & rate  \\ \hline
            0                                                      & 16       & 48      & 8.186e+00 & -    & 2.583e-01 & -     & 6.846e-01 & -     & 4.024e-01 & -     \\
                                                                   & 64       & 192     & 5.679e+00 & 0.53 & 3.371e-01 & -0.38 & 1.035e+00 & -0.60 & 7.088e-01 & -0.82 \\
                                                                   & 256      & 768     & 3.402e+00 & 0.74 & 2.763e-01 & 0.29  & 8.435e-01 & 0.30  & 6.070e-01 & 0.22  \\
                                                                   & 1024     & 3072    & 2.115e+00 & 0.69 & 1.653e-01 & 0.74  & 5.761e-01 & 0.55  & 3.961e-01 & 0.62  \\
                                                                   & 4096     & 12288   & 1.363e+00 & 0.63 & 8.714e-02 & 0.92  & 3.479e-01 & 0.73  & 2.414e-01 & 0.71  \\
                                                                   & 16384    & 49152   & 8.366e-01 & 0.70 & 4.436e-02 & 0.97  & 1.965e-01 & 0.82  & 1.384e-01 & 0.80  \\
                                                                   & 65536    & 196608  & 4.775e-01 & 0.81 & 2.237e-02 & 0.99  & 1.053e-01 & 0.90  & 7.503e-02 & 0.88  \\ \hline
            1                                                      & 16       & 192     & 2.218e+00 & -    & 9.313e-02 & -     & 5.659e-01 & -     & 1.706e-01 & -     \\
                                                                   & 64       & 768     & 1.107e+00 & 1.00 & 4.298e-02 & 1.12  & 2.630e-01 & 1.11  & 8.488e-02 & 1.01  \\
                                                                   & 256      & 3072    & 3.637e-01 & 1.61 & 1.403e-02 & 1.62  & 1.163e-01 & 1.18  & 1.860e-02 & 2.19  \\
                                                                   & 1024     & 12288   & 9.950e-02 & 1.87 & 3.842e-03 & 1.87  & 3.887e-02 & 1.58  & 3.460e-03 & 2.43  \\
                                                                   & 4096     & 49152   & 2.557e-02 & 1.96 & 9.880e-04 & 1.96  & 1.120e-02 & 1.80  & 5.357e-04 & 2.69  \\
                                                                   & 16384    & 196608  & 6.448e-03 & 1.99 & 2.493e-04 & 1.99  & 2.999e-03 & 1.90  & 7.450e-05 & 2.85  \\
                                                                   & 65536    & 786432  & 1.617e-03 & 2.00 & 6.252e-05 & 2.00  & 7.756e-04 & 1.95  & 9.821e-06 & 2.92  \\ \hline
            2                                                      & 16       & 432     & 1.419e+00 & -    & 5.574e-02 & -     & 2.518e-01 & -     & 7.054e-02 & -     \\
                                                                   & 64       & 1728    & 3.199e-01 & 2.15 & 1.234e-02 & 2.18  & 5.277e-02 & 2.25  & 9.511e-03 & 2.89  \\
                                                                   & 256      & 6912    & 4.996e-02 & 2.68 & 1.929e-03 & 2.68  & 8.822e-03 & 2.58  & 5.021e-04 & 4.24  \\
                                                                   & 1024     & 27648   & 6.630e-03 & 2.91 & 2.562e-04 & 2.91  & 1.239e-03 & 2.83  & 2.195e-05 & 4.52  \\
                                                                   & 4096     & 110592  & 8.354e-04 & 2.99 & 3.229e-05 & 2.99  & 1.765e-04 & 2.81  & 8.334e-07 & 4.72  \\
                                                                   & 16384    & 442368  & 1.042e-04 & 3.00 & 4.028e-06 & 3.00  & 2.395e-05 & 2.88  & 2.874e-08 & 4.86  \\
                                                                   & 65536    & 1769472 & 1.299e-05 & 3.00 & 5.022e-07 & 3.00  & 3.119e-06 & 2.94  & 9.433e-10 & 4.93  \\ \hline
        \end{tabular}
        }
    \end{center}
    \caption{Errors and observed convergence rates between the fluxes $\bj$ and $\bJ$ in $L^2(\Omega)$ norm and between the densities $u$ and $U$ in both $L^2(\Omega)$ and $L^\infty(\Omega)$ norms using bi-polynomials $\cQ_k$ with $k=0,1,2$ under a sequence of uniform quadrilateral subdivisions. Optimal rates of convergences are observed. The last column shows the error decay for the cell-wise average of $U$, \ie $\text{avg}:=\max_{K\in\cT} |(u-U,1)_K|$.}
    \label{table:uniform-convergence}
\end{table}

\subsubsection{Convergence for the local postprocessing}
We next examine the convergence of the local postprocessing $U_*$ according to \eqref{eq:l2-minimization}, i.e. reconstruction by the local $L^2$ minimization. Table~\ref{table:l2-minimization} reports the error decays for $U_*$ in both $L^2(\Omega)$ and $L^\infty(\Omega)$ norms. Both errors decay with rate $O(h^{k+2})$ for $k\ge 1$.
\begin{table}[hbt!]
    \begin{center}
        \begin{tabular}{|r|r|r|c|c|c|c|c|c|c|c|c|c|c|c|} \hline
            deg                                             & \# cells & \# dofs &
            \multicolumn{2}{|c|}{$\|u-U_*\|_{L^2(\Omega)}$} &
            \multicolumn{2}{|c|}{$\|u-U_*\|_{L^\infty(\Omega)}$}                                                       \\ \hline
                                                            &          &         & error     & rate & error     & rate \\\hline
            1                                               & 16       & 192     & 7.868e-02 & -    & 5.873e-01 & -    \\
                                                            & 64       & 768     & 2.309e-02 & 1.77 & 2.566e-01 & 1.19 \\
                                                            & 256      & 3072    & 4.557e-03 & 2.34 & 6.750e-02 & 1.93 \\
                                                            & 1024     & 12288   & 7.049e-04 & 2.69 & 1.255e-02 & 2.43 \\
                                                            & 4096     & 49152   & 9.673e-05 & 2.87 & 2.092e-03 & 2.58 \\
                                                            & 16384    & 196608  & 1.260e-05 & 2.94 & 3.007e-04 & 2.80 \\
                                                            & 65536    & 786432  & 1.606e-06 & 2.97 & 4.035e-05 & 2.90 \\ \hline
            2                                               & 16       & 432     & 3.211e-02 & -    & 2.485e-01 & -    \\
                                                            & 64       & 1728    & 3.985e-03 & 3.01 & 5.269e-02 & 2.24 \\
                                                            & 256      & 6912    & 3.647e-04 & 3.45 & 7.371e-03 & 2.84 \\
                                                            & 1024     & 27648   & 2.624e-05 & 3.80 & 5.929e-04 & 3.64 \\
                                                            & 4096     & 110592  & 1.712e-06 & 3.94 & 3.964e-05 & 3.90 \\
                                                            & 16384    & 442368  & 1.085e-07 & 3.98 & 2.511e-06 & 3.98 \\
                                                            & 65536    & 1769472 & 6.813e-09 & 3.99 & 1.574e-07 & 4.00 \\ \hline
        \end{tabular}
    \end{center}
    \caption{Error between the postprocessed solution $U_*$ ($L^2$-minimization) defined by \eqref{eq:l2-minimization} and the exact solution $u$ in both $L^2(\Omega)$ and $L^\infty(\Omega)$ norms using bi-polynomials with degree $1$ and $2$. Optimal rates of convergence are observed.}
    \label{table:l2-minimization}
\end{table}

Table~\ref{table:pde-reconstruction} provides the convergence results for the
flux reconstruction $\bJ_{div}$ defined by \eqref{eq:rtn-proj} as well as the
local postprocessing $U^*$ based on $\bJ_{div}$; see
\eqref{eq:local-postprocessing}. We confirm numerically that both $\bJ_{div}$
and $\DIV\bJ_{div}$ converge with rate $O(h^{k+1})$. Such results, together with
the superconvergence of the cell-wise average of $U$ (see the last column of
Table~\ref{table:uniform-convergence}), guarantee that $\|u-U^*\|_{L^2(\Omega)}$
decays with rate $O(h^{k+2})$ and that the errors are greater than
$\|u-U_*\|_{L^2(\Omega)}$. We also observe optimal convergence of $U^*$ in
$L^\infty(\Omega)$, with an error larger than the $L^\infty(\Omega)$ error for
$U_*$.

\begin{table}[hbt!]
    \begin{center}
    \resizebox{\textwidth}{!}{
        \begin{tabular}{|r|r|r|c|c|c|c|c|c|c|c|c|c|c|c|} \hline
            deg                                                               & \# cells & \# DoFs &
            \multicolumn{2}{|c|}{$\|\bj - \bJ_{div}\|_{L^2(\Omega)}$}         &
            \multicolumn{2}{|c|}{$\|\DIV\bj - \DIV\bJ_{div}\|_{L^2(\Omega)}$} &
            \multicolumn{2}{|c|}{$\|u-U^*\|_{L^2(\Omega)}$}                   &
            \multicolumn{2}{|c|}{$\|u-U^*\|_{L^\infty(\Omega)}$}                                                                                                                  \\ \hline
                                                                              &          &         & error     & rate & error      & rate & error      & rate & error      & rate \\ \hline
            1                                                                 & 16       & 192     & 2.259e+00 & -    & 4.4415e+00 & -    & 8.8866e-02 & -    & 6.1846e-01 & -    \\
                                                                              & 64       & 768     & 1.108e+00 & 1.03 & 2.2154e+00 & 1.00 & 3.0418e-02 & 1.55 & 3.8536e-01 & 0.68 \\
                                                                              & 256      & 3072    & 3.635e-01 & 1.61 & 6.9941e-01 & 1.66 & 6.2003e-03 & 2.29 & 1.6328e-01 & 1.24 \\
                                                                              & 1024     & 12288   & 9.942e-02 & 1.87 & 1.8539e-01 & 1.92 & 9.2468e-04 & 2.75 & 3.3066e-02 & 2.30 \\
                                                                              & 4096     & 49152   & 2.555e-02 & 1.96 & 4.6504e-02 & 2.00 & 1.2232e-04 & 2.92 & 5.0770e-03 & 2.70 \\
                                                                              & 16384    & 196608  & 6.443e-03 & 1.99 & 1.1557e-02 & 2.01 & 1.5584e-05 & 2.97 & 6.9794e-04 & 2.86 \\
                                                                              & 65536    & 786432  & 1.616e-03 & 2.00 & 2.8750e-03 & 2.01 & 1.9622e-06 & 2.99 & 9.1424e-05 & 2.93 \\ \hline
            2                                                                 & 16       & 432     & 1.419e+00 & -    & 3.5109e+00 & -    & 4.6873e-02 & -    & 3.9884e-01 & -    \\
                                                                              & 64       & 1728    & 3.197e-01 & 2.15 & 8.2793e-01 & 2.08 & 7.1029e-03 & 2.72 & 2.3068e-01 & 0.79 \\
                                                                              & 256      & 6912    & 4.992e-02 & 2.68 & 1.3070e-01 & 2.66 & 5.9918e-04 & 3.57 & 2.7253e-02 & 3.08 \\
                                                                              & 1024     & 27648   & 6.624e-03 & 2.91 & 1.7086e-02 & 2.94 & 4.0934e-05 & 3.87 & 2.1080e-03 & 3.69 \\
                                                                              & 4096     & 110592  & 8.347e-04 & 2.99 & 2.1204e-03 & 3.01 & 2.6252e-06 & 3.96 & 1.4343e-04 & 3.88 \\
                                                                              & 16384    & 442368  & 1.041e-04 & 3.00 & 2.6186e-04 & 3.02 & 1.6532e-07 & 3.99 & 9.3157e-06 & 3.94 \\
                                                                              & 65536    & 1769472 & 1.298e-05 & 3.00 & 3.2583e-05 & 3.01 & 1.1081e-08 & 3.90 & 5.9344e-07 & 3.97 \\ \hline
        \end{tabular}
        }
    \end{center}
    \caption{$L^2(\Omega)$-error decay for the flux reconstruction $\bJ_{div}$ (defined by \eqref{eq:rtn-proj}) as well as its divergence. The $L^2(\Omega)$ and $L^\infty(\Omega)$ errors for the postprocessing $U^*$ defined by \eqref{eq:local-postprocessing} (based on flux reconstruction) are report in this table as well.}
    \label{table:pde-reconstruction}
\end{table}

Figure~\ref{fig:sol-compare} reports the approximate solution $U$, its
postprocessing $U_*$ and $U^*$ (top row), and the corresponding error plots on
the uniform grid $\mathcal T_6$ with 16384 cells (bottom row). By examining the
bottom row of the picture, we observe that the error of $U$ is largest near the
two boundary layers but not on the top right corner of the domain. On the other
hand, both $U_*$ and $U^*$ have large errors around the top-right corner.
Meanwhile, cells near the two boundary layers contain both positive and negative
errors, which behave differently compared to the error of $U$.

\begin{figure}[hbt!]
    \begin{center}
        \begin{tabular}{ccc}
            \includegraphics[width=.32\textwidth]{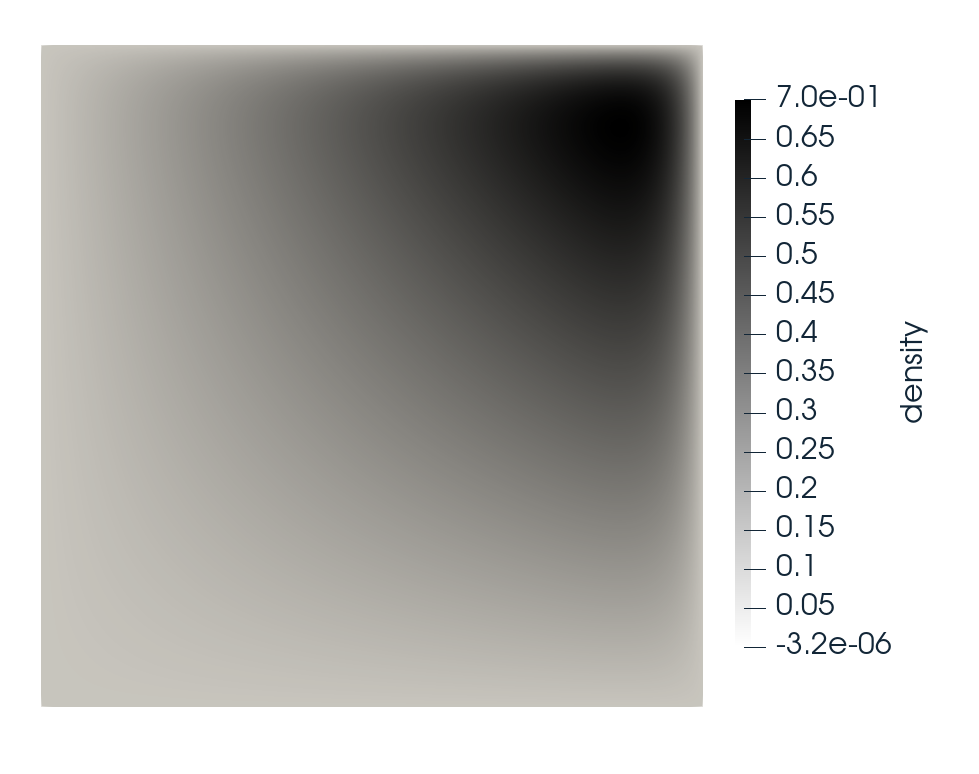}     & \includegraphics[width=.32\textwidth]{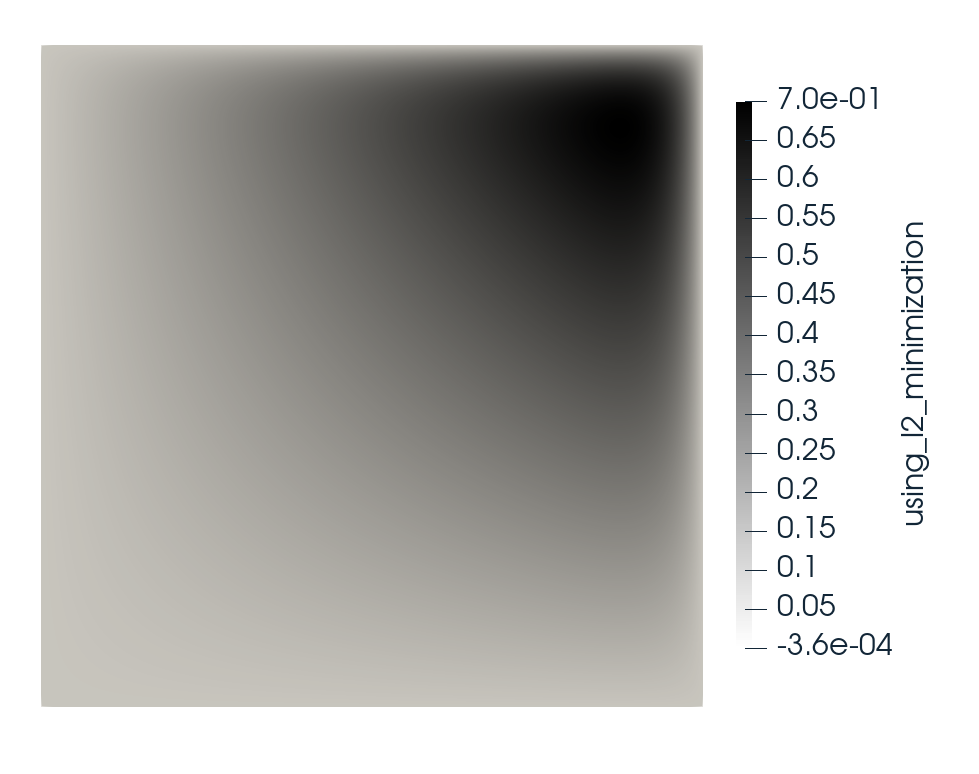}    &
            \includegraphics[width=.32\textwidth]{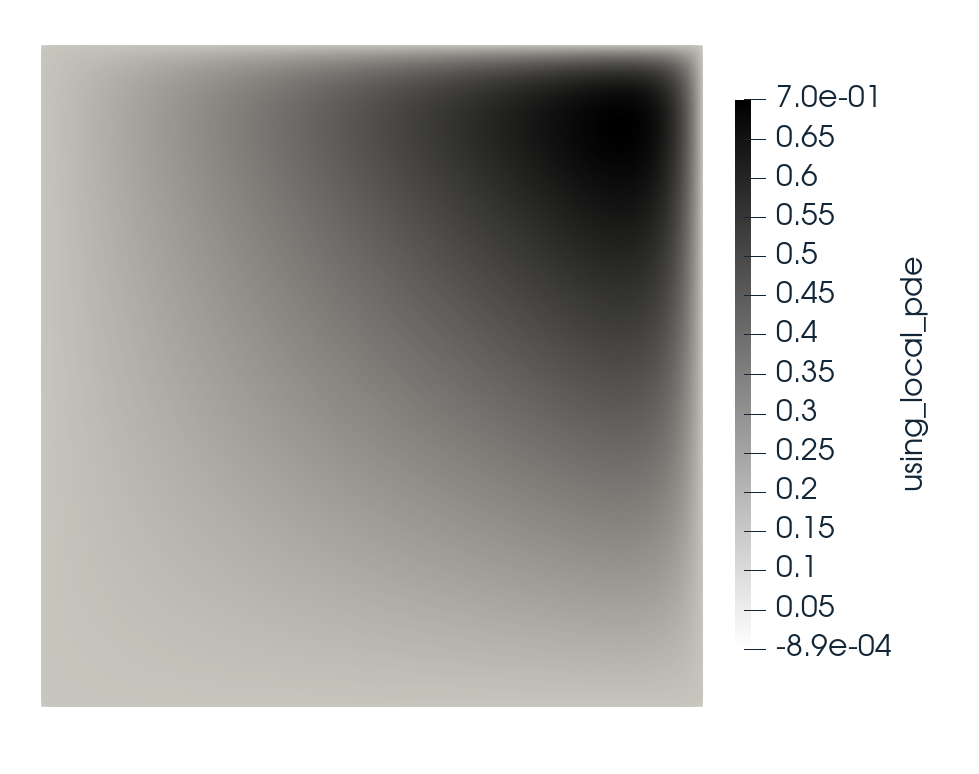}                                                                                                                                   \\
            \includegraphics[width=.32\textwidth]{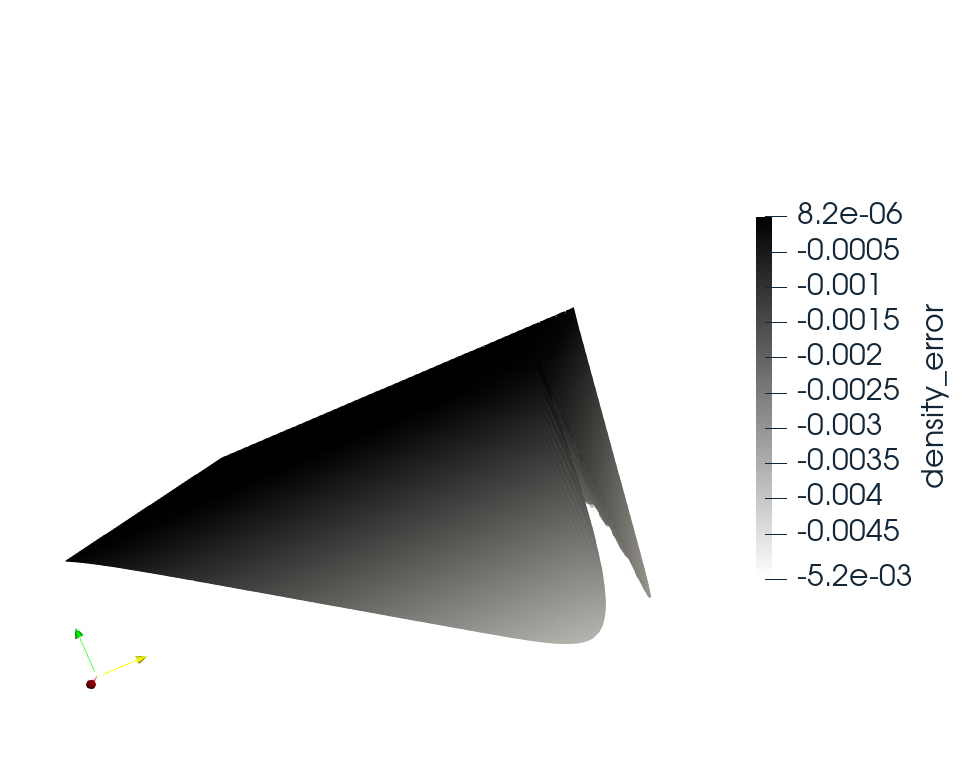} & \includegraphics[width=.32\textwidth]{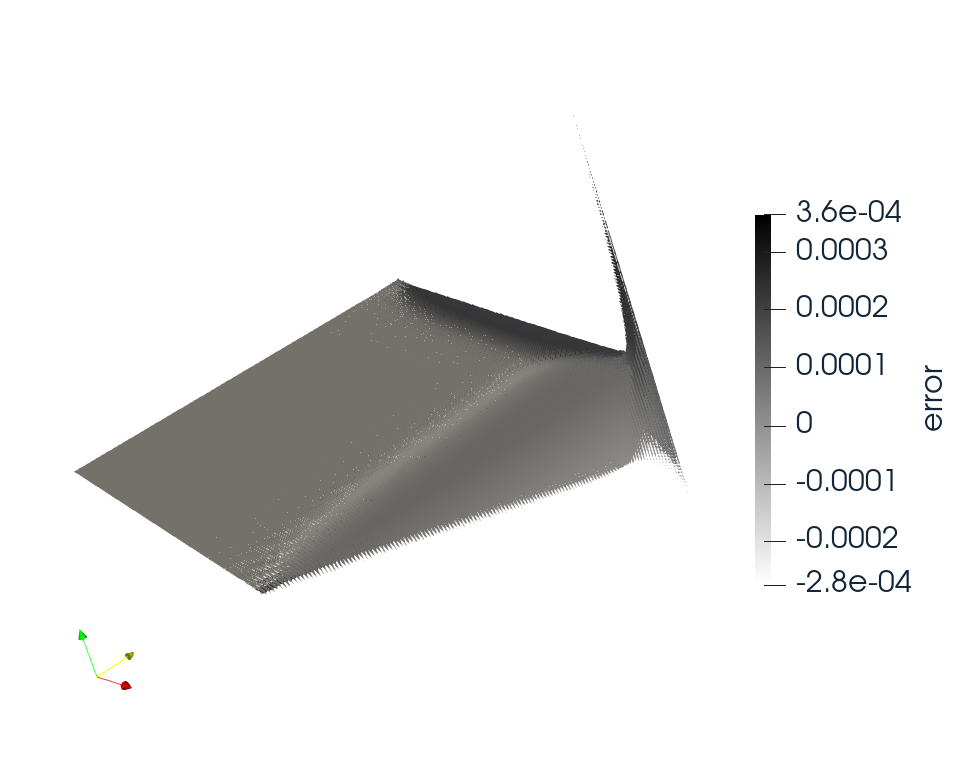} & \includegraphics[width=.32\textwidth]{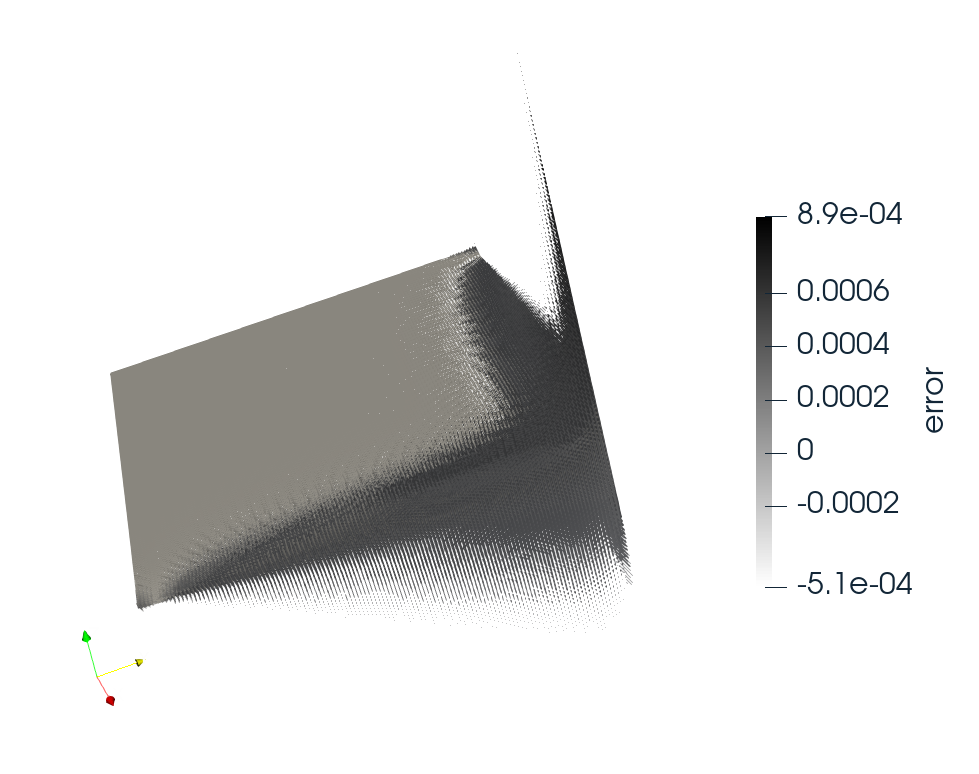}
        \end{tabular}
    \end{center}
    \caption{From left to right: approximate solutions $U$, $U_*$ and $U^*$ (top) defined on the grid $\cT_6$ using bi-linear elements and the corresponding error plots (bottom). The range of the plots are $[-5.2e-3, 8.2e-6]$, $[-2.8e-4,3.6e-4]$ and $[-5.1e-4,8.9e-4]$. Here we rescaled the magnitudes with factors $100$, $1500$, and $1000$ respectively.}
    \label{fig:sol-compare}
\end{figure}

\subsection{A one dimensional benchmark test} \label{ssec:vanroosbroek}

\NewDocumentCommand{\VRGrid}{O{j}}{\ensuremath{\mathcal{T}_{#1}}}

In section~\ref{ssec:limiting-case}, we have pointed out that the
Scharfetter--Gummel scheme \eqref{eq:sg} can be obtained as a certain limit of our proposed method.
Here we perform a comparison test using these two numerical approaches as well as
the standard HDG scheme to solve the stationary van Roosbroeck model in
a one-dimensional domain $\Omega = [0, \ell]$.

\subsubsection{The van Roosbroeck system}
The van Roosbroeck system is one of the most common models used to describe currents of
electric charge carriers inside a semiconductor device. It consists of three
nonlinear differential equations for the unknown electrostatic potential $\psi(x)$, the
density of electrons $n$ and the density of holes $p$. A complete description of
this model can be found in \cite{Markowich}. In the following numerical test, we compute
the so-called thermodynamic equilibrium, a
physical state defined by vanishing currents. In equilibrium, it is possible to decouple
the above mentioned three equations, i.e., solving a nonlinear Poisson equation to obtain
$\psi(x)$ and using this function to compute the other two unknown densities $(n,p)$.

Here the one dimensional Poisson equation for $\psi(x)$ is given by
\begin{equation} \label{eq:theq-psi}
  -\frac{\mathrm{d}}{\mathrm{d}x}\left(\varepsilon_s \frac{\mathrm{d}}{\mathrm{d}x} \psi \right) =
      q \left( N_v \exp\left(\frac{E_v - q \psi}{k_B T}\right)
          - N_c \exp\left(\frac{q \psi - E_c}{k_B T}\right) + C
      \right),
\end{equation}
where $q$ is the elementary charge, $k_B$ is the Boltzmann constant, $T$
represents the temperature of the device (which we will suppose to be constant on all
$\Omega$), and $\varepsilon_s$, $N_v$, $E_v$, $N_c$, $E_c$, and $C$ are
constants or piecewise constant functions that describe some physical properties of the
material; we refer to  Table~\ref{tab:const-values} for all these values.
Dirichlet  boundary conditions for \eqref{eq:theq-psi} can be obtained under the assumption
that the right hand side datum vanishes at the boundary (cf. \cite[Section~2.3]{Markowich}).

Given the electrostatic potential $\psi$, we shall solve for the densities $n$
and $p$ governed by the van Roosbroeck system. Here we can solve them separately
in thermal equilibrium and we only focus on the hole density. For the hole
density, there holds
\begin{equation} \label{eq:theq-p}
  \begin{cases}
    \displaystyle -q \frac{\mathrm{d}}{\mathrm{d}x}\left(
        \mu_p\left(U_T \frac{\mathrm{d}p}{\mathrm{d}x} +
           p\frac{\mathrm{d}\psi}{\mathrm{d}x}\right)\right) = 0 , \\[12pt]
    \displaystyle p(x_0) = N_v \exp\left(\frac{E_v - q \psi\left(x_0\right)}{k_BT}\right) &
      \text{for } x_0 \in \{0, \ell\} .
  \end{cases}
\end{equation}
where $U_T := \tfrac{k_BT}{q}$ and the mobility constant $\mu_p$ defines how easily the
holes move through the device.
Comparing the above equation with \eqref{eq:problem}, we have
\[
  \alpha := U_T, \qquad \qquad
  \beta := -\frac{\mathrm{d}\psi}{\mathrm{d}x}, \qquad \qquad
  f := 0.
\]

\subsubsection{Numerical settings}
We will test our numerical method for \eqref{eq:theq-p} using the electrostatic
potential given by \eqref{eq:theq-psi}. Here we consider our physical domain as
a p-i-n device, which contains $p$ and an $n$ doped regions as well as an
intrinsic layer in between. The physical meaning of these terms is for example explained in \cite{Farrell2017c,Farrell2017}.  The doping concentration $C$ in \eqref{eq:theq-psi}
is assumed to be a piecewise constant function
\begin{equation}
 C = \begin{cases}
      N_D, & \text{if } 0 \leq x < \tfrac{\ell}{3}, \\
      \, 0, & \text{if } \tfrac{\ell}{3} \leq x < \tfrac{2 \ell}{3}, \\
      - N_A, & \text{if } x \geq \tfrac{2 \ell}{3},
     \end{cases}
\end{equation}
where $N_D$ and $N_A$ are positive constants defined in
Table~\ref{tab:const-values} and $\ell>0$ is taken to be a third of the domain
length. A graphical representation of the solutions to this problem is shown in
Figure~\ref{fig:van-roosbroeck-solution}.

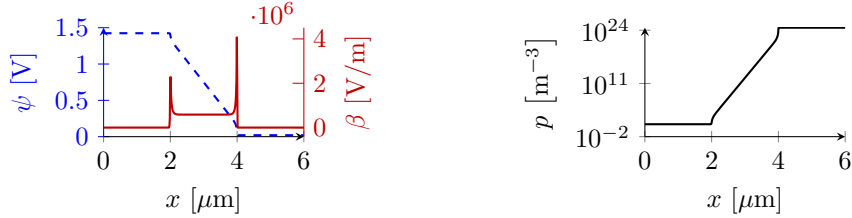
\begin{figure}[hbt]
 \centering

 \pgfplotstableread{data/van_roosbroeck_solution.dat}{\vrsolution}
 \begin{tikzpicture}
  \pgfplotsset{
    y axis style/.style={
      yticklabel style=#1,
      ylabel style=#1,
      y axis line style=#1,
      ytick style=#1
    }
  }
  \begin{axis}[%
      width=.35\textwidth,
      height=.25\textwidth,
      axis x line=bottom,
      axis y line=left,
      y axis style=blue,
      xmin=0, xmax=6,
      ymin=0, ymax=1.5,
      ylabel = {$\psi$}, y unit={V},%y SI prefix=centi,
      xlabel = {$x$}, x SI prefix=micro, x unit={m}
    ]
    \addplot[blue, dashed, thick] table[x=x, y=V]{\vrsolution};
  \end{axis}
    \begin{axis}[%
      width=.35\textwidth,
      height=.25\textwidth,
      axis x line=none,
      axis y line*=right,
      xmin=0, xmax=6,
      y axis style=red!75!black,
      ylabel = {$\beta$}, y unit={V/m},
    ]
    \addplot[red!75!black, thick] table[x=x, y=E]{\vrsolution};
  \end{axis}
 \end{tikzpicture}
 \hspace{45pt}
 \begin{tikzpicture}
  \begin{axis}[%
      width=.35\textwidth,
      height=.25\textwidth,
      axis x line=bottom,
      axis y line=left,
      xmin=0, xmax=6,
      ymin=1e-2,
      ymode=log,
      ylabel = {$p$}, y unit={m^{-3}},%y SI prefix=centi,
      xlabel = {$x$}, x SI prefix=micro, x unit={m}
    ]
    \addplot[black, thick] table[x=x, y=p]{\vrsolution};
  \end{axis}
 \end{tikzpicture}
 \caption{%
   (Left) electrostatic potential $\psi$ (dashed blue line) for
   \eqref{eq:theq-psi} and the electric field $\beta=-\diff \psi/\diff x$ (solid
   red line) for a model that uses the constants described in
   Table~\ref{tab:const-values}. (Right) the hole concentration for \eqref{eq:theq-p}.
   Here we obtain the results using linear elements on a highly refined grid
   with 18432 cells and 92161 degrees of freedom (counting the degrees of
   freedom for the solution, its gradient and the solution on the trace). }
 \label{fig:van-roosbroeck-solution}
\end{figure}

We construct a sequence of hierarchical subdivision $\VRGrid[i]$ of  
$\Omega$ for
$i$ in $\{1, \ldots, 7\}$. The coarsest grid $\VRGrid[1]$ is a refinement of the
uniform partition of $\Omega$ with intervals $I_j:=(\tfrac{(j-1)\ell}{6},
\tfrac{j\ell}{6})$ for $j=1,\ldots,6$, such that the two junctions $\{\tfrac\ell3,
\tfrac{2\ell}3\}$ are grid points of $\VRGrid[i]$. Then we build $\VRGrid[1]$ as
follows:
\begin{enumerate}
    \item We set $I_1, I_6\in \VRGrid[1]$.
    \item Bisect the rest of the intervals with the following graded strategy:
    for each $I_j$ with $j=2,3,4,5$, we consider an affine transformation
    $\mathcal F_j : I_j \to (0,1)$ such that the junction point is at $0$. Then
    we subdivide $(0,1)$ uniformly with four intervals and rescale them by
    taking the square of the position, i.e., for $i=1$ the partition is given by
    $\mathcal P
    :=\{(\tfrac{(k-1)^2}{2^{i+2}},\tfrac{k^2}{2^{i+2}})\}_{k=1}^{2^{i+1}}$.
    \item $\VRGrid[1]$ consists of $I_1$, $I_6$, and $\mathcal F^{-1}_j(\mathcal P)$ for $j=2,3,4,5$.
\end{enumerate}
To generate $\VRGrid[i]$ for $i>1$, we refine globally  $\{I_1, I_6\}$ with $i$ times,
and refine the other $I_j$ intervals following Step~2. Examples of the grids in
$\{\VRGrid[i]\}_{i=1}^3$ are shown in Figure~\ref{fig:triangulation}.

\begin{figure}[hbt]
  \centering
  \tikzsetnextfilename{van_roosbroeck_grids}
  \begin{tikzpicture}
    \def\domainlength{6}
    \def\dsize{.9\textwidth}
    \def\coordtick{0.2}
    \def\levonetick{1.4}
    \def\levtwotick{0.7}
    \def\levthreetick{0.2}

    \draw[thick] (0,0) node[anchor=north east]{0} -- (\dsize, 0) node[anchor=north west] {$\ell$};
    \draw[black, thick] (\dsize/3, 0)--(\dsize/3, -\coordtick) node[anchor=north]{$\frac{\ell}{3}$};
    \draw[black, thick](2*\dsize/3,0)--(2*\dsize/3,-\coordtick)node[anchor=north]{$\frac{2\ell}{3}$};

    \foreach \x in {1.0, 1.4375, 1.75, 1.9375, 2.0, 2.0625, 2.25, 2.5625, 3.0, 3.4375, 3.75,%
                    3.9375, 4.0, 4.0625, 4.25, 4.5625, 5.0} {
      \def\xp{\x / \domainlength * \dsize}
      \draw[black, thick] (\xp, 0) -- (\xp, \levonetick);
    }

    \foreach \x in {0.5, 1.234375, 1.609375, 1.859375, 1.984375, 2.015625, 2.140625, 2.390625, %
                    2.765625, 3.234375, 3.609375, 3.859375, 3.984375, 4.015625, 4.140625, 4.390625, %
                    4.765625, 5.5}
    {
      \def\xp{\x / \domainlength * \dsize}
      \draw[black, thick] (\xp, 0) -- (\xp, \levtwotick);
    }

    \foreach \x in {0.25, 0.75, 1.12109375, 1.33984375, 1.52734375, 1.68359375, 1.80859375, %
                    1.90234375, 1.96484375, 1.99609375, 2.00390625, 2.03515625, 2.09765625, %
                    2.19140625, 2.31640625, 2.47265625, 2.66015625, 2.87890625, 3.12109375, %
                    3.33984375, 3.52734375, 3.68359375, 3.80859375, 3.90234375, 3.96484375, %
                    3.99609375, 4.00390625, 4.03515625, 4.09765625, 4.19140625, 4.31640625, %
                    4.47265625, 4.66015625, 4.87890625, 5.25, 5.75}
    {
      \def\xp{\x / \domainlength * \dsize}
      \draw[black, thick] (\xp, 0) -- (\xp, \levthreetick);
    }

    \end{tikzpicture}
  \caption{%
    The triangulation $\mathcal T_j$ over $\Omega$ with 18, 36 and 72 cells.
  }
  \label{fig:triangulation}
\end{figure}
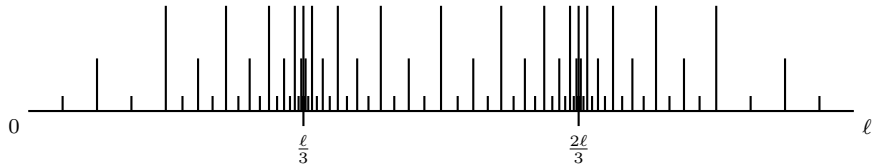

We approximate the solutions to \eqref{eq:theq-p} on $\VRGrid[j]$ with three
different methods: the Scharfetter--Gummel scheme \eqref{eq:sg}, the standard
HDG method (H-LDG) \cite{cockburn2009unified} and our proposed one in
\eqref{eq:discrete-full}. Denote the approximation of the hole density with
$P_j$. We obtain a numerical reference solution $P_r$ by solving the problem
with the standard HDG method on a grid $\mathcal T_r$ which is globally refined
from $\mathcal T_7$ four times. The coefficient $\beta =
-\frac{\mathrm{d}\psi}{\mathrm{d}x}$ is generated by solving the Poisson problem
\eqref{eq:theq-psi} on $\mathcal T_r$ with the standard HDG and project (in the
$L^2$ sense) the resulting $\beta$ onto $\bcQ(\mathcal T)$ using piecewise constant functions.

\begin{figure}
  \centering
  \def\pwidth{.95\textwidth}
  \def\pheight{.75\textwidth}
  \begin{subfigure}[b]{0.48\textwidth}
    \centering
    \tikzsetnextfilename{van_roosbroeck_comparison-no-postproc}
    \begin{tikzpicture}
      \pgfplotstableread{%
          data/van_roosbroeck_method_comparison_beta0.dat%
      }{\methodsbetazero}
      \begin{axis}[%
          width=\pwidth,
          height=\pheight,
          axis x line=bottom,
          axis y line=left,
          xmode=log,
          ymode=log,
          xmin=16, xmax=1200,
          ymin=1e-5,
          xtick=data,
          xticklabel={%
              \pgfkeys{/pgf/number format/precision={0}}
              \pgfmathparse{exp(\tick)}
              \pgfmathprintnumber{\pgfmathresult}
          },
          ylabel = {}, % y unit={m^{-3}}, y SI prefix=centi,
          xlabel = {N. of cells},
          legend pos=south west,
          legend style={font=\tiny}
        ]

        \addplot[black, mark=*] table[x=x, y=FVM]{\methodsbetazero};
        \addlegendentry{FVM}
        \addplot[blue, mark=diamond*] table[x=x, y=HDG_deg0]{\methodsbetazero};
        \addlegendentry{HDG}
        \addplot[orange, mark=square*] table[x=x, y=WHDG_deg0]{\methodsbetazero};
        \addlegendentry{WHDG}

        \pgfplotstablegetelem{0}{x}\of\methodsbetazero
        \let\xs=\pgfplotsretval
        \pgfplotstablegetelem{0}{FVM}\of\methodsbetazero
        \let\ys=\pgfplotsretval
        \pgfkeys{/pgf/fpu=true}
        \pgfmathsetmacro{\rone}{\xs * \ys}
        \pgfmathsetmacro{\rtwo}{\xs * \xs * \ys}
        \pgfkeys{/pgf/fpu=false}

        \edef\draworder{%
            \noexpand\addplot[domain=18:1152, samples=2, gray, dashed] {\rone / x};
            \noexpand\addlegendentry{Slope=-1}
            \noexpand\addplot[domain=18:1152, samples=2, gray, dotted, thick] {\rtwo / x^2};
            \noexpand\addlegendentry{Slope=-2}
        }
        \draworder
      \end{axis}
    \end{tikzpicture}
  \end{subfigure}
  \hfill
  \begin{subfigure}[b]{0.48\textwidth}
    \centering
    \tikzsetnextfilename{van_roosbroeck_comparison-no-postproc-solution}
    \begin{tikzpicture}
      \pgfkeys{/pgf/fpu=true}
      \pgfplotstableread{%
          data/van_roosbroeck_solutions_0144.dat%
      }{\solutions}
      \pgfkeys{/pgf/fpu=false}

      \pgfplotstablegetrowsof{\solutions}
      \pgfmathtruncatemacro{\NCells}{\pgfplotsretval}
      \pgfmathtruncatemacro{\LastCell}{\NCells - 1}

      \pgfplotstablegetelem{0}{FVM_x_start}\of\solutions
      \let\domainStart=\pgfplotsretval

      \pgfplotstablegetelem{\LastCell}{FVM_x_end}\of\solutions
      \let\domainEnd=\pgfplotsretval

      \begin{axis}[%
          width=\pwidth,
          height=\pheight,
          axis x line=bottom,
          axis y line=left,
          xmin=\domainStart,
          xmax=\domainEnd,
          ymin=1e-2,
          ymax=1e25,
          ymode=log,
          ylabel = {$p$}, y unit={m^{-3}}, %y SI prefix=centi,
          xlabel = {$x$}, x SI prefix=micro, x unit={m},
        ]

        \NewDocumentCommand{\DrawPiecewiseFunction}{O{black}m}{%
          \foreach \i in {0, ..., \LastCell}%
          {
            \pgfplotstablegetelem{\i}{#2_x_start}\of\solutions
            \let\xs=\pgfplotsretval
            \pgfplotstablegetelem{\i}{#2_x_end}\of\solutions
            \let\xe=\pgfplotsretval
            \pgfplotstablegetelem{\i}{#2_y_start}\of\solutions
            \let\ys=\pgfplotsretval
            \pgfplotstablegetelem{\i}{#2_y_end}\of\solutions
            \let\ye=\pgfplotsretval
            \edef\drawcell{
              \noexpand\draw [#1, thick] (\xs, \ys) -- (\xe, \ye);
            }
            \drawcell
          }
        }
        \DrawPiecewiseFunction{FVM}
        \DrawPiecewiseFunction[blue]{HDG_deg0}
        \DrawPiecewiseFunction[orange]{WHDG_deg0}
      \end{axis}
    \end{tikzpicture}
  \end{subfigure}
  \begin{subfigure}[b]{0.48\textwidth}
    \centering
    \tikzsetnextfilename{van_roosbroeck_comparison-with-postproc}
    \begin{tikzpicture}
      \pgfplotstableread{%
          data/van_roosbroeck_method_comparison_beta0.dat%
      }{\methodsbetazero}

      \begin{axis}[%
        width=\pwidth,
        height=\pheight,
        axis x line=bottom,
        axis y line=left,
        xmode=log,
        ymode=log,
        xmin=16, xmax=1200,
        ymin=1e-5,
        xtick=data,
        xticklabel={%
            \pgfkeys{/pgf/number format/precision={0}}
            \pgfmathparse{exp(\tick)}
            \pgfmathprintnumber{\pgfmathresult}
        },
        ylabel = {}, % y unit={m^{-3}}, y SI prefix=centi,
        xlabel = {N. of cells},
        legend pos=south west,
        legend style={font=\tiny}
      ]

        \addplot[black, mark=*] table[x=x, y=FVMp]{\methodsbetazero};
        \addlegendentry{FVM}

        \addplot[blue, mark=diamond*] table[x=x, y=HDGpt]{\methodsbetazero};
        \addlegendentry{HDG}

        \addplot[orange, mark=square*] table[x=x, y=WHDGpt]{\methodsbetazero};
        \addlegendentry{WHDG}
        
        % \addplot[green, mark=square*] table[x=x, y=HDGpt]{\methodsbetazero};
        % \addlegendentry{HDGpt}

        % \addplot[red, mark=square*] table[x=x, y=WHDGpt]{\methodsbetazero};
        % \addlegendentry{WHDGpt}

        % Draw the lines of the order of convergence
        \pgfplotstablegetelem{0}{x}\of\methodsbetazero
        \let\xs=\pgfplotsretval
        \pgfplotstablegetelem{0}{FVM}\of\methodsbetazero
        \let\ys=\pgfplotsretval
        \pgfkeys{/pgf/fpu=true}
        \pgfmathsetmacro{\rone}{\xs * \ys}
        \pgfmathsetmacro{\rtwo}{\xs * \xs * \ys}
        \pgfkeys{/pgf/fpu=false}

        \edef\draworder{%
            \noexpand\addplot[domain=18:1152, samples=2, gray, dashed] {\rone / x};
            \noexpand\addlegendentry{Slope=-1}
            \noexpand\addplot[domain=18:1152, samples=2, gray, dotted, thick] {\rtwo / x^2};
            \noexpand\addlegendentry{Slope=-2}
        }
        \draworder

      \end{axis}
    \end{tikzpicture}
  \end{subfigure}
  \hfill
  \begin{subfigure}[b]{0.48\textwidth}
    \centering
    \tikzsetnextfilename{van_roosbroeck_comparison-with-postproc-solution}
    \begin{tikzpicture}
      \pgfkeys{/pgf/fpu=true}
      \pgfplotstableread{%
          data/van_roosbroeck_solutions_0072.dat%
      }{\solutions}
      \pgfkeys{/pgf/fpu=false}

      \pgfplotstablegetrowsof{\solutions}
      \pgfmathtruncatemacro{\NCells}{\pgfplotsretval}
      \pgfmathtruncatemacro{\LastCell}{\NCells - 1}

      \pgfplotstablegetelem{0}{FVM_x_start}\of\solutions
      \let\domainStart=\pgfplotsretval

      \pgfplotstablegetelem{\LastCell}{FVM_x_end}\of\solutions
      \let\domainEnd=\pgfplotsretval

      \begin{axis}[%
          width=\pwidth,
          height=\pheight,
          axis x line=bottom,
          axis y line=left,
          xmin=\domainStart,
          xmax=\domainEnd,
          ymin=1e-2,
          ymax=1e25,
          ymode=log,
          ylabel = {$p$}, y unit={m^{-3}}, %y SI prefix=centi,
          xlabel = {$x$}, x SI prefix=micro, x unit={m},
        ]

        \NewDocumentCommand{\DrawPiecewiseFunction}{O{black}m}{%
          \foreach \i in {0, ..., \LastCell}%
          {
            \pgfplotstablegetelem{\i}{#2_x_start}\of\solutions
            \let\xs=\pgfplotsretval
            \pgfplotstablegetelem{\i}{#2_x_end}\of\solutions
            \let\xe=\pgfplotsretval
            \pgfplotstablegetelem{\i}{#2_y_start}\of\solutions
            \let\ys=\pgfplotsretval
            \pgfplotstablegetelem{\i}{#2_y_end}\of\solutions
            \let\ye=\pgfplotsretval

            \pgfkeys{/pgf/fpu=true}
            \pgfmathsetmacro{\m}{(\ye - \ys) / (\xe - \xs)}
            \pgfkeys{/pgf/fpu=false}

            \edef\drawcell{
              \noexpand\addplot[domain=\xs:\xe, samples=50, #1, thick]{%
                  \ys + \m * (x - \xs)
              };
            }

            \pgfkeys{/pgf/fpu=true}
            \pgfmathparse{abs(\m) > 1e-5 ? 1 : 0}
            \pgfkeys{/pgf/fpu=false}

            \pgfmathfloattofixed{\pgfmathresult}
            \pgfmathparse{\pgfmathresult > 0.5 ? int(1) : int(0)}

            \ifthenelse{\pgfmathresult=1}{%
              \drawcell
            }{%
              \edef\drawflatcell{
                \noexpand\draw [#1, thick] (\xs, \ys) -- (\xe, \ys);
              }
              \drawflatcell
            }
          }
        }

        \NewDocumentCommand{\DrawTruncatedPiecewiseFunction}{O{black}m}{%
          \foreach \i in {0, ..., \LastCell}%
          {
            \pgfplotstablegetelem{\i}{#2_x_start}\of\solutions
            \let\xs=\pgfplotsretval
            \pgfplotstablegetelem{\i}{#2_x_end}\of\solutions
            \let\xe=\pgfplotsretval
            \pgfplotstablegetelem{\i}{#2_y_start}\of\solutions
            \let\ys=\pgfplotsretval
            \pgfplotstablegetelem{\i}{#2_y_end}\of\solutions
            \let\ye=\pgfplotsretval

            \pgfkeys{/pgf/fpu=true}
            \pgfmathsetmacro{\m}{(\ye - \ys) / (\xe - \xs)}
            \pgfkeys{/pgf/fpu=false}

            \edef\drawcell{
              \noexpand\addplot[domain=\xs:\xe, samples=50, #1, thick]{%
                  max(\ys + \m * (x - \xs), 1e-3)
              };
            }

            \pgfkeys{/pgf/fpu=true}
            \pgfmathparse{abs(\m) > 1e-5 ? 1 : 0}
            \pgfkeys{/pgf/fpu=false}

            \pgfmathfloattofixed{\pgfmathresult}
            \pgfmathparse{\pgfmathresult > 0.5 ? int(1) : int(0)}

            \ifthenelse{\pgfmathresult=1}{%
              \drawcell
            }{%
              \edef\drawflatcell{
                \noexpand\draw [#1, thick] (\xs, \ys) -- (\xe, \ys);
              }
              \drawflatcell
            }
          }
        }

        \DrawPiecewiseFunction{FVMp}
        \DrawTruncatedPiecewiseFunction[blue]{HDGpt}
        \DrawPiecewiseFunction[orange]{WHDGpt}
        % \DrawPiecewiseFunction[green]{WHDGpt}
        % \DrawTruncatedPiecewiseFunction[red]{HDGpt}
      \end{axis}
    \end{tikzpicture}
  \end{subfigure}
  \caption{%
    (Top-left) $L^2(\Omega)$-errors between $P_j$ and $P_r$ against the number cells using different numerical approaches. (Top-right) Numerical approximation of $P_4$. (Bottom-left) post-processing $P_j^*$ of $P_j$ using linear elements. (Bottom-right) Post-processing results for $P_3^*$. Here we have truncated the results $P_3^*$ from the standard HDG for the values above $10^{-2}\unit{\per\cubic\meter}$.
  }
  \label{fig:van-roosbroeck-method-comparison}

\end{figure}
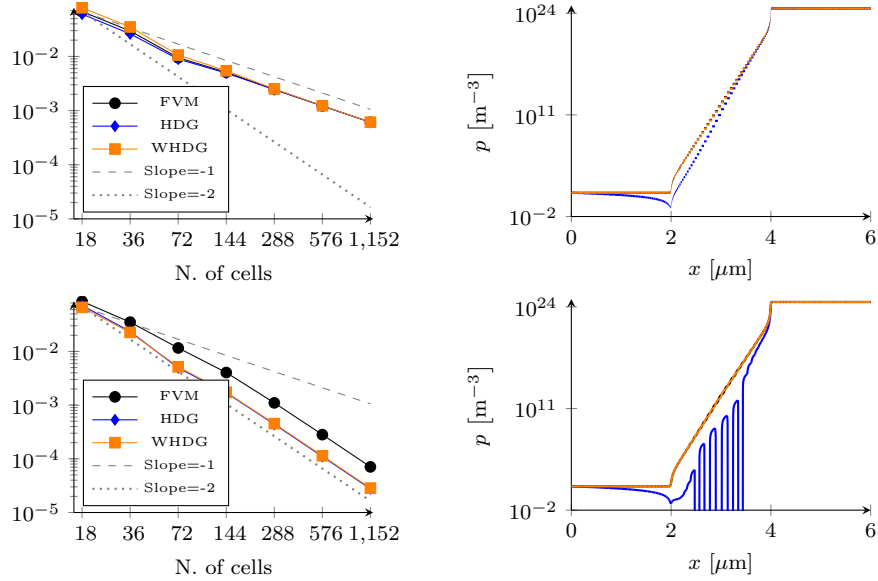

\subsubsection{Results}
In the top-left plot of Figure~\ref{fig:van-roosbroeck-method-comparison}, we
report the errors $\|P_j - P_r\|_{L^2(\Omega)}$ versus the numbers of degrees of
freedom, where, to maintain the comparison fair with the finite volume scheme,
we use $\mathcal P_0$ functions for the HDG and W-HDG approaches. It turns out
that all three numerical methods converge to first order, however, the numerical
solution computed using the standard HDG violates the maximum principle; see the
top-right plot for the profile of $P_4$. On the other hand, the numerical
results using \eqref{eq:discrete-full} behaves similarly to the state-of-the-art
solution obtained with Scharfetter-Gummel stabilized finite volumes. 

It is common to postprocess the solution generated by the Scharfetter-Gummel
stabilized finite volume scheme by connecting the mid-points of all cells with
straight lines, obtaining a piecewise linear and continuous solution. Therefore, in
the bottom part of Figure~\ref{fig:van-roosbroeck-method-comparison}, we perform an
analogous comparison after implementing this strategy for all three methods.
In particular, for the HDG and W-HDG scheme, this algorithm has been implemented
not by connecting the mid-points of the cells but the points on the trace obtained
from $\hat{u}$. The bottom-left logarithmic plot shows that the resulting linear
approximations, denoted by $P_j^*$, converge to $P_r$ with second order in all cases.
On the other side, the quality of the solution on the points where the charge carrier
density is low (i.e., between 0 and $\tfrac{2}{3}\ell$) is quite different between
the FVM and W-HDG method and the standard HDG; indeed, the contribution 
of these regions
on the global $L^2$ error is negligible.
On the bottom-right plot we can see the profile for $P_3^*$. The approximations from the
FVM method and the W-HDG both satisfy the discrete maximum principle; instead, the result
for the standard HDG method does not. Some values of $\hat{u}$ are negative (up
to $\num[round-mode=places,round-precision=3]{-2.28426488521920e+13} \unit{\per\cubic\meter}$);
therefore, here we only show the part with values above $10^{-2}\unit{\per\cubic\meter}$.

\begin{table}
  \centering
  \sisetup{round-mode = figures, round-precision = 5}
  \renewcommand{\arraystretch}{0.7}
  \begin{tabular}{l|c|c|c}
    Physical Quantity & Symbol & Value & Units\\
    \hline
    Device length & $\ell$ & 6.00 & \unit{\micro\meter} \\
    Reference temperature & $T$ & \num{300} & \unit{\kelvin}\\
    Absolute permittivity & $\varepsilon_s$ &
         \num{1.14219022847298e-10} & \unit{\farad\per\meter}\\
    Valence band density & $N_v$ &\num{9.139615903601645e24} & \unit{\per\cubic\meter}\\
    Valence band-edge energy & $E_v$ & \num{0} & \unit{\eV} \\
    Conduction band density & $N_c$ & \num{4.351959895879690e23} & \unit{\per\cubic\meter}\\
    Conduction band-edge energy & $E_c$ & \num{1.424} & \unit{\eV} \\
    Hole mobility & $\mu_p$ & \num{4e-2} & \unit{\meter\squared\per\volt\per\second} \\
    Acceptor doping density & $N_A$ & \num{4.204223315656757e24} & \unit{\per\cubic\meter} \\
    Donator doping density & $N_D$ & \num{4.351959895879690e23} & \unit{\per\cubic\meter} \\
    Elementary charge & $q$ & \num{1.602176634e-19} & \unit{\coulomb} \\
    Boltzmann constant & $k_B$ & \num{1.380649e-23} & \unit{\joule\per\kelvin}
  \end{tabular}
  \caption{%
      The physical constants used in Section~\ref{ssec:vanroosbroek}.
  }
  \label{tab:const-values}
\end{table}
\ignorespaces

\section{Conclusion}\label{sec:conclusion}

In this work we presented a weighted HDG method for the solution of
convection-diffusion problems. The method is based on the introduction of
weighted scalar products when defining the local HDG equations. The additional
weights have the effect of a local Slotboom change of variable: they eliminate
locally the drift term, transforming the local drift diffusion problems into
diffusion problems with an exponential coefficient.

We prove well-posedness  property of the method,
validate its numerical stability as well as convergence properties.
Furthermore, we provide a fair comparison with a
state-of-the-art finite volume discretization using the Scharfetter-Gummel
flux approximation, as well as standard HDG methods, for a convection dominated problem
with dramatic scale changes, derived from modeling a p-i-n semiconductor
junction.

W-HDG is equivalent to the Scharfetter and Gummel finite volume method for
piecewise constant polynomial approximations, and generalizes it to arbitrary
high order, while maintaining similar stability properties. An important path of
investigation would be the construction of a positivity preserving
post-processing solution for W-HDG, which is not guaranteed from the current
algorithm.

%\begin{acknowledgements}
%\end{acknowledgements}


\paragraph{Data Availability}
Enquiries about data availability should be directed to the authors.

\section*{Declarations}
\paragraph{Conflict of interest}

The authors declare that they have no conflict of interest.

\bibliographystyle{spmpsci}      % mathematics and physical sciences
\bibliography{references}

% \begin{thebibliography}{}
% %
% % and use \bibitem to create references. Consult the Instructions
% % for authors for reference list style.
% % Format for Journal Reference
% Author, Article title, Journal, Volume, page numbers (year)
% Author, Book title, page numbers. Publisher, place (year)
% \end{thebibliography}

\end{document}